\documentclass[pdftex,a4paper,12pt]{scrartcl}

\usepackage[T1]{fontenc}
\usepackage{lmodern} 
\usepackage{exscale}

\usepackage[pdftex]{graphicx}
\usepackage{amsmath,amssymb,amsthm}
\usepackage{mathtools}
\usepackage{hyperref}
\usepackage[capitalize]{cleveref}
\usepackage{enumitem}

\usepackage[en-US]{datetime2}

\usepackage{macros}
\usepackage[tikzset={x=.7cm,y=.7cm}]{linkcobs} 

\ifdefined\cprime\else
\def\cprime{'}
\fi

\theoremstyle{plain}
\newtheorem{mainthm}{Main Theorem}
\newtheorem*{maincor}{Corollary}
\newtheorem{theorem}{Theorem}[section]
\newtheorem{proposition}[theorem]{Proposition}
\newtheorem{corollary}[theorem]{Corollary}
\newtheorem{lemma}[theorem]{Lemma}

\theoremstyle{definition}
\newtheorem{definition}[theorem]{Definition}
\newtheorem{problem}[theorem]{Problem}

\theoremstyle{remark}
\newtheorem{example}[theorem]{Example}
\newtheorem{remark}[theorem]{Remark}
\newtheorem{notation}[theorem]{Notation}
\newtheorem*{convention}{Convention}

\numberwithin{equation}{section}
\numberwithin{figure}{section}

\title{A cobordism realizing crossing change on $\mathfrak{sl}_2$ tangle homology and a categorified Vassiliev skein relation}
\date{\DTMdisplaydate{2020}{7}{31}{-1}}
\author{Noboru Ito and Jun Yoshida}

\begin{document}
\maketitle

\begin{abstract}
In this paper, we discuss degree $0$ crossing change on Khovanov homology in terms of cobordisms.
Namely, using Bar-Natan's formalism of Khovanov homology, we introduce a sum of cobordisms that yields a morphism on complexes of two diagrams of crossing change, which we call the ``genus-one morphism.''    
It is proved that the morphism is invariant under the moves of double points in tangle diagrams.
As a consequence, in the spirit of Vassiliev theory, taking iterated mapping cones, we obtain an invariant for singular tangles that extending sl(2) tangle homology; examples include Lee homology, Bar-Natan homology, and Naot's universal Khovanov homology as well as Khovanov homology with arbitrary coefficients.
We also verify that the invariant satisfies categorified analogues of Vassiliev skein relation and the FI relation.
\end{abstract}


\tableofcontents

\section{Introduction}
\label{sec:intro}

The goal of this paper is to realize crossing change on $\mathfrak{sl}_2$ homology, aka Khovanov homology, in the spirit of Vassiliev theory.
More precisely, following Bar-Natan's formalism~\cite{BarNatan2005}, we discuss a formal sum of two cobordisms.
We prove that it induces a morphism on Khovanov complexes which is of bidegree $(0,0)$ and invariant under moves of double points (\cref{main:genus1}).
Thanks to this invariance, the morphism extends Khovanov homology and its variants to invariants of singular links (\cref{main:UKH-sing}); the examples include Lee homology~\cite{Lee2005}, Bar-Natan homology~\cite{BarNatan2005}, and Naot's universal Khovanov homology~\cite{Naot2006}.
It is also proved that the extended invariants satisfy a categorified form of the \emph{FI relation}, which is one of the fundamental relations appearing in Kontsevich's universal construction of Vassiliev invariants~\cite{Kontsevich1993}.

In the work of Vassiliev~\cite{Vassiliev1990}, it was pointed out that crossing change can be understood as a ``wall-crossing'' in the space of knots $\mathcal K$.
This viewpoint was formulated more clearly by Birman and Lin~\cite{Birman1993,BirmanLin1993}; they pointed out that Vassiliev's finite order invariants are characterized in terms of vanishing of their extension to singular knots via the following the \emph{Vassiliev skein relation}:
\begin{equation}
\label{eq:intro:Vassiliev-skein}
v^{(r+1)}\left(\diagSingUp\right)
= v^{(r)}\left(\diagCrossPosUp\right)
- v^{(r)}\left(\diagCrossNegUp\right)
\quad.
\end{equation}
The extended invariant $v^{(r)}$ is sometimes called the \emph{$r$-th Vassiliev derivative} of $v$, and $v$ is a Vassiliev invariant if and only if $v^{(r)}\equiv 0$ for sufficiently large $r$.
Notice that, as suggested in \eqref{eq:intro:Vassiliev-skein}, in Vassiliev theory, double points are regarded as traces of crossing change and measure the effect of them.
In fact, they showed a strong relationship between Vassiliev invariants and \emph{quantum invariants}: if $v$ is a quantum invariant with values in Laurent polynomials, then in its Taylor expansion, the coefficient of order $r$ is a Vassiliev invariant of order $r$.

Although the result of Birman and Lin is suggestive, the relation is, however, unclear in case of link homologies.
In an attempt to establish it, the first problem we have to solve is the following.

\begin{problem}\label{ProbCatVassiliev}
Realize the crossing change in \eqref{eq:intro:Vassiliev-skein} as a morphism on link homologies so that it is invariant under moves of singular knots.
\end{problem}

One may notice that an instance of crossing change is provided in link cobordism theory; it is realized as the following composition of cobordisms:
\begin{equation}
\label{eq:intro:cob-crossingchange}
\diagCrossNegUp{}=
\begin{tikzpicture}[baseline=-.5ex]
\node[circle,inner sep=2] (L) at (-.2,0) {};
\node[circle,inner sep=2] (R) at (.2,0) {};
\draw[red,very thick,-stealth] (-.6,-.8) -- (L) (L) to[out=56,in=180] (R.north) to[out=0,in=236] (.6,.8);
\draw[red,very thick,-stealth] (.6,-.8) to[out=124,in=0] (R.south) to[out=180,in=-56] (L.center) -- (-.6,.8);
\end{tikzpicture}
\xrightarrow{\text{saddle}}
\begin{tikzpicture}[baseline=-.5ex]
\node[circle,inner sep=2] (L) at (-.2,0) {};
\node[circle,inner sep=1] (R) at (.2,0) {};
\draw[red,very thick,-stealth] (-.6,-.8) -- (L) (L) to[out=56,in=90,looseness=3] (R.west) to[out=-90,in=-56,looseness=3] (L.center) -- (-.6,.8);
\draw[red,very thick,-stealth] (.6,-.8) to[out=124,in=-90] (R.east) to[out=90,in=236] (.6,.8);
\end{tikzpicture}
\xrightarrow{\RMove1^2}
\begin{tikzpicture}[baseline=-.5ex]
\node[circle,inner sep=2] (L) at (-.2,0) {};
\node[circle,inner sep=1] (R) at (.2,0) {};
\draw[red,very thick,-stealth] (-.6,-.8) -- (L.center) to[out=56,in=90,looseness=3] (R.west) to[out=-90,in=-56,looseness=3] (L) (L) -- (-.6,.8);
\draw[red,very thick,-stealth] (.6,-.8) to[out=124,in=-90] (R.east) to[out=90,in=236] (.6,.8);
\end{tikzpicture}
\xrightarrow{\text{saddle}}
\begin{tikzpicture}[baseline=-.5ex]
\node[circle,inner sep=2] (L) at (-.2,0) {};
\node[circle,inner sep=2] (R) at (.2,0) {};
\draw[red,very thick,-stealth] (-.6,-.8) -- (L.center) to[out=56,in=180] (R.north) to[out=0,in=236] (.6,.8);
\draw[red,very thick,-stealth] (.6,-.8) to[out=124,in=0] (R.south) to[out=180,in=-56] (L) (L) -- (-.6,.8);
\end{tikzpicture}
=\diagCrossPosUp
\quad.
\end{equation}
By the functoriality of Khovanov homology~\cite{Jacobsson2004,ClarkMorrisonWalker2009}, the sequence~\eqref{eq:intro:cob-crossingchange} gives rise to a map between Khovanov homologies of the form
\begin{equation}
\label{eq:intro:cob-crossingchange-Kh}
\mathit{Kh}\left(\diagCrossNegUp\right)
\to \mathit{Kh}\left(\diagCrossPosUp\right)
\quad.
\end{equation}
Unfortunately, it turns out that this is not a suitable instance of crossing change with respect to Vassiliev theory.
Indeed, a computation shows that the map~\eqref{eq:intro:cob-crossingchange-Kh} is of bidegree $(0,2)$ in the standard gradings, and, as a consequence, \eqref{eq:intro:cob-crossingchange-Kh} is not an isomorphism even if the crossing change actually does nothing.
In fact, it was showed by Hedden and Watson~\cite{HeddenWatson2018} that the map~\eqref{eq:intro:cob-crossingchange-Kh} induces a categorified version of \emph{Jones skein relation}, not Vassiliev skein relation~\eqref{eq:intro:Vassiliev-skein}.

As an attempt to finding a degree $(0,0)$ crossing change, we note that we have obtained an answer to \cref{ProbCatVassiliev} in our previous paper~\cite{ItoYoshida2019} in the case of Khovanov homology with coefficients in the field $\mathbb{F}_2$ of two elements.
Indeed, we constructed a chain map
\begin{equation}\label{eq:intro:PhiHatF2}
\widehat\Phi:
C^{\ast,\star}\left(\diagCrossNegUp;\mathbb F_2\right)
\to C^{\ast,\star}\left(\diagCrossPosUp;\mathbb F_2\right)
\quad,
\end{equation}
called the \emph{genus-one morphism}, so that it is invariant under moves of singular links.
More precisely, we showed that it is invariant under the moves
\begin{equation}
\label{eq:intro:moveSing}
\diagCrossSingRivOL \leftrightarrow \diagCrossSingRivOR
\ ,\quad
\diagCrossSingRivUL \leftrightarrow \diagCrossSingRivUR
\ ,\quad
\diagRvSingU \leftrightarrow \diagRvSingD
\quad,
\end{equation}
which generate all the moves of singular links.
On the other hand, this construction does not work for general coefficients or variants of Khovanov homologies, including Lee homology and Bar-Natan homology.

We achieve the goal by extending this result to these variants.
For this, we use Bar-Natan's category of ``picture''~\cite{BarNatan2005} with a little bit different notations to emphasize the relation to topological field theories (cf.~\cite{LaudaPfeiffer2009}).
Namely, let $k$ be a fixed coefficient ring.
For compact oriented $0$-dimensional manifolds $Y_0$ and $Y_1$, let us denote by $k\mathbf{Cob}_2(Y_0,Y_1)$ the $k$-linear category with cobordisms $Y_0\to Y_1$ as objects and formal sums of \emph{$2$-bordisms} (\cite{SchommerPries2009}, aka~\emph{cobordisms with corners}~\cite{Laures2000}) between them with coefficients in $k$ as morphisms.
Bar-Natan introduced three relations on the category called \ref{relK:S}, \ref{relK:T}, and \ref{relK:4Tu} relations; the quotient category will be denoted by $\Cob(Y_0,Y_1)$ in this paper.
He then constructed a chain complex $\KhOf{D}$ in $\Cob(\partial^-D,\partial^+D)$ for each tangle diagram $D$ and proved the following theorem.

\begin{theorem}[{Bar-Natan~\cite[Theorem~1]{BarNatan2005}}]
\label{theo:UKH-BarNatan}
The homotopy type of $\KhOf{D}$ is invariant under Reidemeister moves so that it defines an invariant of tangles.
\end{theorem}

To recover link homologies, one can apply a functor $Z:\Cob(\varnothing,\varnothing)\to\mathbf{Mod}_k$ to the complex $\KhOf{D}$ for a link diagram $D$.
Indeed, by virtue of \cref{theo:UKH-BarNatan}, the homology of the resulting chain complexes over $k$ is an invariant for links.
In particular, in the case where $Z$ is the TQFT associated with the Frobenius algebra $k[x]/(x^2)$, the construction above yields the original Khovanov homology.
As for variants, one can take the Frobenius algebra $C_{h,t}\coloneqq k[x]/(x^2-hx-t)$ for $h,t\in k$ (see e.g.~\cite{Khovanov2006}, \cite{LaudaPfeiffer2009}); if $(h,t)=(0,1),(1,0)$, one obtains Lee homology~\cite{Lee2005} and Bar-Natan homology~\cite{BarNatan2005} respectively.
This type of Frobenius algebras are comprehensively discussed in \cite{Khovanov2006}, and some sorts of universality are proved in \cite{Naot2006}.

We generalize the \emph{genus-one morphism} $\widehat\Phi$ in \eqref{eq:intro:PhiHatF2} to the universal Khovanov homology based on the morphism given by the following sum of cobordisms:
\[
\BordPhiFst\;-\;\BordPhiSnd:\diagSmoothV\to\diagSmoothV
\quad.
\]
In fact, it turns out that this induces a morphism of chain complexes
\begin{equation}
\label{eq:intro:UKH-genus1}
\widehat\Phi:\KhOf*{\diagCrossNegUp}\to\KhOf*{\diagCrossPosUp}
\quad.
\end{equation}
Our first main result is the following.

\begin{mainthm}[\cref{theo:genus1-inv}, \cref{theo:Phi-comm-RI}]
\label{main:genus1}
The genus-one morphism $\widehat\Phi$ is of bidegree $(0,0)$ and satisfies the following properties:
\begin{enumerate}[label=\upshape(\arabic*)]
  \item it is invariant under the moves in \eqref{eq:intro:moveSing}.
  \item it commutes with the morphisms associated with the Reidemeister moves of type \RomNum1.
\end{enumerate}
\end{mainthm}

As a result, we obtain a map realizing crossing change on link homologies.
In particular, it agrees with the genus-one morphism constructed in~\cite{ItoYoshida2019} in the case of Khovanov homology with coefficients in $\mathbb F_2$.

By virtue of Vassiliev skein relation~\eqref{eq:intro:Vassiliev-skein}, crossing change gives rise to an extension of link invariants to singular links.
In the case of link homologies, an analogous idea works; indeed, we consider mapping cones instead of the substitution in \eqref{eq:intro:Vassiliev-skein}.

\begin{mainthm}[\cref{cor:vassiliev-cone-UKH}, \cref{theo:Khsing-inv}]
\label{main:UKH-sing}
For every singular tangle diagram $D$, there is a complex $\KhOf{D}$ in the category $\Cob(\partial_0D,\partial_1D)$ so that there is an isomorphism
\begin{equation}
\label{eq:intro:cat-Vassiliev-skein}
\KhOf*{\diagSingUp}
\cong\Cone\left(\KhOf*{\diagCrossNegUp}\xrightarrow{\widehat\Phi}\KhOf*{\diagCrossPosUp}\right)
\quad.
\end{equation}
Furthermore, the complex $\KhOf{D}$ is invariant under moves of singular tangles up to chain homotopy equivalences.
\end{mainthm}

We note that the isomorphism~\eqref{eq:intro:cat-Vassiliev-skein} categorifies the Vassiliev skein relation~\eqref{eq:intro:Vassiliev-skein}.
Indeed, it induces the following long exact sequence of Khovanov homologies with coefficients in $k$:
\[
\begin{tikzcd}[column sep=.9em]
\cdots \ar[r] & \operatorname{Kh}^{i,j}\left(\diagCrossNegUp\;;k\right) \ar[r,"\widehat\Phi_\ast"] & \operatorname{Kh}^{i,j}\left(\diagCrossPosUp\;;k\right) \ar[r] & \operatorname{Kh}^{i,j}\left(\diagSingUp\;;k\right) \ar[dl,out=-20,in=160] & \\
&& \operatorname{Kh}^{i+1,j}\left(\diagCrossNegUp\;;k\right) \ar[r,"\widehat\Phi_\ast"] & \operatorname{Kh}^{i+1,j}\left(\diagCrossPosUp\;;k\right) \ar[r] &
\cdots\quad.
\end{tikzcd}
\]
Taking the Euler characteristics, one recovers the Vassiliev skein relation on the (unnormalized) Jones polynomial.
Hence, this observation also shows that our extension of Khovanov homology categorifies Jones polynomial even on singular links.

As an invariant of singular knots in view of Vassiliev theory, the extended Khovanov homology enjoys several properties.
We in particular focus on the \emph{FI relation}, which appears in Kontsevich's construction of the universal Vassiliev invariant~\cite{Kontsevich1993}; namely, it is easily checked that for every knot invariant $v$, its Vassiliev derivatives satisfy the following identity:
\begin{equation}
\label{eq:intro:FI-abinv}
v^{(r)}\left(\diagFiSing\right) = 0
\quad.
\end{equation}
To understand it more clearly, we quickly review the work of Vassiliev~\cite{Vassiliev1990}.
Let us denote by $\mathcal M$ the space of generic smooth maps $S^1\to\mathbb R^3$ equipped with Whitehead $C^\infty$-topology.
Then, Thom-Boardman theory (see~\cite[Chapter~VI]{GolubitskyGuillemin1973}) gives rise to a stratification $\mathcal M=\bigcup_i\mathcal M_i$: for example,
\begin{itemize}
  \item $\mathcal M_0$ consists of smooth embeddings;
  \item $\mathcal M_1$ consists of smooth immersions with exactly one double point;
  \item $\mathcal M_2$ consists of
\begin{enumerate}[label=\upshape(\alph*)]
  \item\label{M2:crit} smooth injections with exactly one critical point and
  \item smooth immersion with exactly two double points.
\end{enumerate}
\end{itemize}
Although $\mathcal M$ is not finite dimensional, it turns out that each stratum $\mathcal M_i\subset\mathcal M$ has codimension exactly $i$.
Since knot invariants can be seen as cohomology classes on $\mathcal M$, ``Poincar\'e duality'' hence implies that the homology class of $\mathcal M_i$ yields degree $i$ relations.
For instance, Vassiliev skein relation comes from the stratum $\mathcal M_1$.

The FI relation~\eqref{eq:intro:FI-abinv} is also one of them.
Indeed, for a point in $\mathcal M_2$ of type~\ref{M2:crit} above, its neighborhood in $\mathcal M$ is depicted as in \cref{fig:FI-Vassiliev}, and let us consider the two paths (1) and (2) there.
\begin{figure}[tbp]
\centering
\begin{tikzpicture}
\node[below] (N) at (0,-1) {\diagFiNeg};
\node[above] (P) at (0,1) {\diagFiPos};
\node[right] (Cr) at (1,0) {\diagFiCrit};
\node[right] (Non) at (3,0) {\diagFiNil};
\draw[red] (-3,0) node[left]{\diagFiSing} -- node[below left] {$\mathcal M_1$} (1,0);
\fill[blue] (1,0) circle (.15) node[below]{$\mathcal M_2$};
\draw[thick,dotted] (N) to[bend right] (Non);
\draw[thick,dotted,-stealth] (Non) to[bend right] node[above right]{(1)} (P);
\draw[thick,dotted,-stealth] (N) to[bend left] node[above left] {(2)} (P);
\end{tikzpicture}
\caption{The FI~relation in Vassilev theory}
\label{fig:FI-Vassiliev}
\end{figure}
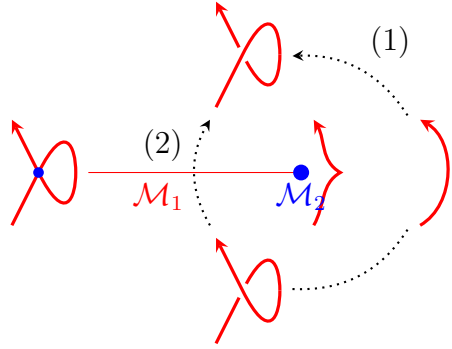
If a knot $K$ moves along the path (1), the value of $v(K)$ does not change.
On the other hand, if it goes along (2), $v(K)$ is subject to Vassiliev skein relation since it crosses the ``wall'' $\mathcal M_1$.
Comparing the effects of the two paths, we obtain the equation \eqref{eq:intro:FI-abinv}.
Actually, our extension of Khovanov complex satisfies a categorified analogue.

\begin{maincor}[\cref{cor:Ubrc-FI}]
The complex below is contractible, i.e. the identity is null-homotopic
\[
\KhOf*{\diagFiSing}
\quad.
\]
\end{maincor}

To the best of our knowledge, there is no singular tangle homology except ours known to satisfy the FI relation in the sense above.
In other words, it distinguishes our invariant from the others.

The plan of this paper is as follows.
We review Bar-Natan's complexes of cobordisms for ordinary tangles in \cref{sec:univ-Kh}.
In particular, the category $\Cob(Y_0,Y_1)$ is defined in terms of cobordisms with corners; we mainly follow~\cite{SchommerPries2009} for this material.
We also define the \emph{universal bracket complex} as the ``unshifted'' version of the universal Khovanov complex.
The checkerboard colorings are discussed to determine the orientations on cobordisms.

We then define the genus-one morphism $\widehat\Phi$ \eqref{eq:intro:UKH-genus1} in \cref{sec:genus-1}.
In \cref{sec:invariance}, we show the invariance of $\widehat\Phi$ under the moves of singular tangles.
Using the result, in \cref{sec:UKHsing}, we extend the Khovanov complex to singular tangles and verify the FI relation.

\section{The universal Khovanov  complex}\label{sec:univ-Kh}

\subsection{Cobordisms of manifolds with corners}

In order to develop tangle homology, we need the notion of cobordism of manifolds with corners.
We here give a brief sketch, and for details, we refer the reader to 
\cite[Definition~1]{Janich1968}, \cite{Laures2000}, and \cite{SchommerPries2009}.

Let $Y_0$ and $Y_1$ be closed oriented $0$-manifolds (i.e.~finite sets with a label $\{-,+\}$ on each element), and let $W_0$ and $W_1$ be two cobordisms from $Y_0$ to $Y_1$.
Then, a \emph{$2$-bordism} from $W_0$ to $W_1$ is a compact oriented $2$-manifold $S$ with corners such that
\begin{itemize}
\item $\partial S$ is a union of submanifolds: $\partial S=\partial_0 S\cup\partial_1 S$;
\item it is equipped with orientation-preserving diffeomorphisms
\[
s_0:\overbar{W_0}\amalg W_1 \xrightarrow\simeq \partial_0S
\ ,\quad
s_1: (\overbar{Y_0}\amalg Y_1)\times[0,1] \xrightarrow\simeq \partial_1S
\quad,
\]
here $\overbar{W_0}$ and $\overbar{Y_0}$ are respectively  the manifolds $W_0$ and $Y_0$ with the reversed orientation.  
\end{itemize}
In this case, we write $S:W_0\to W_1:Y_0\to Y_1$ or $S:W_0\to W_1$ simply.

\begin{definition}\label{def:Cob}
For closed oriented $0$-manifolds $Y_0$ and $Y_1$, we define a category $\mathbf{Cob}_{2}(Y_0,Y_1)$ as follows:
\begin{itemize}
  \item the objects are cobordisms $W:Y_0\to Y_1$;
  \item the morphisms are diffeomorphism classes of $2$-bordisms $S:W_0\to W_1:Y_0\to Y_1$, where only diffeomorphisms that preserve orientations and structure maps are considered;
  \item the composition is given by gluing.
\end{itemize}
\end{definition}

\begin{remark}
\label{rem:glue-unique}
By Collar Neighborhood Theorem \cite[Lemma~2.1.6]{Laures2000},   every chain of bordisms actually  admits gluing.
It also turns out that gluing is unique up to diffeomorphisms of bordisms.
In general, though a choice of such diffeomorphisms is not canonical, it can be done within a canonical choice of an isotopy class \cite{SchommerPries2009}.
\end{remark}

In view of \cref{rem:glue-unique}, the composition is associative.
For a cobordism $W:Y_0\to Y_1$, the identity on $W$ is represented by the trivial $2$-bordism $W\times[0,1]$.
Hence, $\mathbf{Cob}_{2}(Y_0,Y_1)$ is a category.

\begin{convention}
In this paper, we always use the ``bottom-to-top'' convention for cobordisms and the ``left-to-right'' one for $2$-bordisms as in \cref{fig:example-2bord}.
\begin{figure}[htbp]
\centering
\begin{tikzpicture}[xlen=.5pt,ylen=-.5pt,scale=.8]
\node at (0,0) {$S$};
\node at (-100,0) {$\partial_1^- S$};
\node at (100,0) {$\partial_1^+ S$};
\node[left] at (0,60) {$\xrightarrow{\simeq}\;\partial_0^-S$};
\node[right] at (0,-60) {$\partial_0^+S\;\xleftarrow{\simeq}$};
\begin{scope}[transform canvas={xshift=-120}]
\node at (-100,0) {$Y_0$};
\node at (100,0) {$Y_1$};
\node at (0,60) {$W_0$};
\draw[red,very thick] (-59.5,40.5) -- (-19.5,40.5) to[quadratic={(0.5,40.5)}] (0.5,20.5) to[quadratic={(0.5,0.5)}] (-19.5,0.5) to[quadratic={(-39.5,0.5)}] (-39.5,-19.5) to[quadratic={(-39.5,-39.5)}] (-19.5,-39.5) -- (60.5,-39.5);
\draw[red,very thick] (60.5,40.5) to[quadratic={(40.5,40.5)}] (40.5,20.5) to[quadratic={(40.5,0.5)}] (60.5,0.5);
\fill[blue] (-59.5,40.5) circle (.5ex);
\fill[blue] (60.5,-39.5) circle(.5ex);
\fill[blue] (60.5,40.5) circle(.5ex);
\fill[blue] (60.5,0.5) circle(.5ex);
\end{scope}
\draw[red,very thick] (-73.59,48.63) -- (-33.59,48.63) to[quadratic={(-13.59,48.63)}] (-6.545,44.57) to[quadratic={(-4.71,43.51)}] (-4.71,42.73);
\draw[red,very thick,dotted] (-4.71,42.73) to[quadratic={(-4.71,40.5)}] (-19.5,40.5) to[quadratic={(-39.5,40.5)}] (-32.46,36.43) to[quadratic={(-25.41,32.37)}] (-5.41,32.37) -- (-3.76,32.37);
\draw[red,very thick] (-3.76,32.37) -- (29.86,32.37);
\draw[red,very thick,dotted] (29.86,32.37) -- (60.5,32.37);
\draw[red,very thick] (60.5,32.37) -- (74.59,32.37);
\draw[red,very thick] (46.41,48.63) to[quadratic={(31.62,48.63)}] (31.62,46.41);
\draw[red,very thick,dotted] (31.62,46.41) to[quadratic={(31.62,45.63)}] (33.46,44.57) to[quadratic={(37.64,42.15)}] (46.41,41.17);
\draw[red,very thick] (46.41,41.17) to[quadratic={(52.39,40.5)}] (60.5,40.5);
\draw[black] (-4.71,42.72) .. controls+(0,-40)and+(0,-40) .. (31.62,46.41);
\draw[black,thick,dotted] (-34.29,38.28) -- (-34.29,-31.37);
\draw[black] (-34.29,-31.37) -- (-34.29,-41.72);
\draw[black] (9.38,-45.41) .. controls(9.38,-40.08)and(10.03,-35.4) .. (11.14,-31.37);
\draw[black,thick,dotted] (11.14,-31.37) .. controls(17.61,-8.11)and(39.9,-7) .. (44.76,-31.37);
\draw[black] (44.76,-31.37) .. controls(45.37,-34.42)and(45.71,-37.87) .. (45.71,-41.72);
\draw[blue,very thick] (-73.59,48.63) -- (-73.59,-31.37);
\draw[blue,very thick] (46.41,48.63) -- (46.41,-31.37);
\draw[blue,very thick] (60.5,40.5) -- (60.5,-39.5);
\draw[blue,very thick] (74.59,32.37) -- (74.59,-47.63);
\draw[red,very thick] (-73.59,-31.37) -- (46.41,-31.37);
\draw[red,very thick] (-32.46,-43.57) to[quadratic={(-39.5,-39.5)}] (-19.5,-39.5) to[quadratic={(0.5,-39.5)}] (7.545,-43.57) to[quadratic={(14.59,-47.63)}] (-5.41,-47.63) to[quadratic={(-25.41,-47.63)}] (-32.46,-43.57);
\draw[red,very thick] (60.5,-39.5) to[quadratic={(40.5,-39.5)}] (47.54,-43.57) to[quadratic={(54.59,-47.63)}] (74.59,-47.63);
\begin{scope}[transform canvas={xshift=120}]
\node at (-100,0) {$Y_0$};
\node at (100,0) {$Y_1$};
\node at (0,-60){$W_1$};
\draw[red,very thick] (-59.5,40.5) -- (60.5,40.5);
\draw[red,very thick] (-39.5,-19.5) to[quadratic={(-39.5,0.5)}] (-19.5,0.5) to[quadratic={(0.5,0.5)}] (0.5,-19.5) to[quadratic={(0.5,-39.5)}] (-19.5,-39.5) to[quadratic={(-39.5,-39.5)}] (-39.5,-19.5);
\draw[red,very thick] (60.5,0.5) to[quadratic={(40.5,0.5)}] (40.5,-19.5) to[quadratic={(40.5,-39.5)}] (60.5,-39.5);
\fill[blue] (-59.5,40.5) circle (.5ex);
\fill[blue] (60.5,-39.5) circle(.5ex);
\fill[blue] (60.5,40.5) circle(.5ex);
\fill[blue] (60.5,0.5) circle(.5ex);
\end{scope}
\end{tikzpicture}
\caption{Example of a $2$-bordism (orientation omitted).}
\label{fig:example-2bord}
\end{figure}
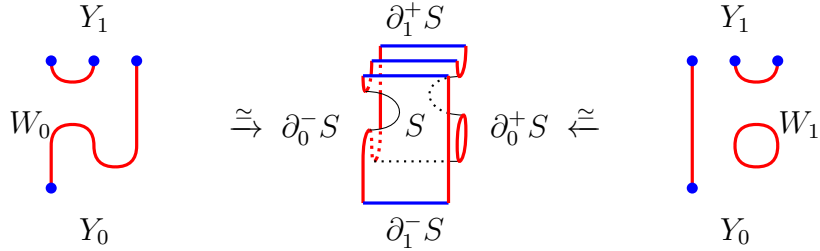
\end{convention}

The gluing also gives rise to a functor.
Indeed, if $Y_0$, $Y_1$ and $Y_2$ are closed oriented $0$-manifolds, then we define a functor
\begin{equation}
\label{eq:h-composition}
(\blank)\ast(\blank):\mathbf{Cob}_{2}(Y_1,Y_2)\times\mathbf{Cob}_{2}(Y_0,Y_1)
\to \mathbf{Cob}_{2}(Y_0,Y_2)
\end{equation}
as follows (see \cite{SchommerPries2009} for details):
\begin{itemize}
  \item for objects, if $W:Y_0\to Y_1$ and $W':Y_1\to Y_2$, then choose a gluing $\widetilde W$ of $W$ and $W'$ along $Y_1$, and set $W'\ast W\coloneqq \widetilde W$ (which is often denoted by $W'\circ W$);
  \item for morphisms, if $S:W_0\to W_1:Y_0\to Y_1$ and $S':W'_0\to W'_1:Y_1\to Y_2$, then $S\ast S'$ is the diffeomorphism class of a gluing of $S$ and $S'$ along $Y_1$ with respect to $W'_0\circ W_0$ and $W'_1\circ W_1$.
\end{itemize}

\begin{lemma}[\cite{SchommerPries2009}]
\label{lem:h-comp-unique}
The construction above actually defines a unique functor \eqref{eq:h-composition} up to a canonical isomorphism.
\end{lemma}

The disjoint union of manifolds gives rise to another functor
\begin{equation}
\label{eq:cobfunc-disj}
(\blank)\otimes(\blank):\mathbf{Cob}_{2}(Y_0,Y_1)\times\mathbf{Cob}_{2}(Y'_0,Y'_1)
\to \mathbf{Cob}_{2}(Y_0\amalg Y'_0,Y_1\amalg Y'_1)
\quad.
\end{equation}
This functor is associative in the sense that it defines an essentially unique functor
\[
\begin{multlined}
\mathbf{Cob}_{2}(Y^{(1)}_0,Y^{(1)}_1)\times\dots\times\mathbf{Cob}_{2}(Y^{(r)}_0,Y^{(r)}_1)
\\
\to \mathbf{Cob}_{2}(Y^{(1)}_0\amalg\dots\amalg Y^{(r)}_0,Y^{(1)}_1\amalg\dots\amalg Y^{(r)}_1)
\quad.
\end{multlined}
\]
In particular, in the case $Y_0=Y_1=\varnothing$, we obtain a symmetric monoidal structure on the category $\mathbf{Cob}_{2}(\varnothing,\varnothing)$.
We further introduce two functors, both of which are given by the orientation reversion:
\begin{equation}
\label{eq:cobfunc-rev}
\begin{gathered}
\begin{array}[t]{rccc}
  \rho_0:&\mathbf{Cob}_{2}(Y_0,Y_1)&\to&\mathbf{Cob}_{2}(\overbar{Y_0},\overbar{Y_1})\\
  \text{on cobordisms} & W &\mapsto& \overbar{W} \\
  \text{on $2$-bordisms} & S &\mapsto& \overbar{S}
\end{array}
\quad,\\[2ex]
\begin{array}[t]{rccc}
  \rho_2:&\mathbf{Cob}_{2}(Y_0,Y_1)^{\opposite}&\mapsto&\mathbf{Cob}_{2}(Y_0,Y_1)\\
  \text{on cobordisms} & W &\mapsto& W \\
  \text{on $2$-bordisms} & S &\mapsto& \overbar{S}
\end{array}
\quad.
\end{gathered}
\end{equation}
These functors respect gluing and disjoint union.

\subsection{The category $\Cob(Y_0, Y_1)$}\label{sec:univ-Kh:K}

Let $k$ be a commutative ring and $\mathcal{C}$ a $k$-linear category.  We define a category $Mat(\mathcal{C})$ as follows \cite[VIII.2, Exercise~6]{MacLane1971}:
\begin{itemize}
\item An object is a tuple $(A_1, A_2, \ldots, A_n)$, which is denoted by $\bigoplus_{i=1}^{n} A_i$ of $n \in \mathbb{N}$ and $A_i \in \mathcal{C}$.
\item For objects $\bigoplus_{i=1}^{n} A_i$ and $\bigoplus_{j=1}^{m} B_j$, a morphism is defined by the set $\{ f_{ij} : A_i \to B_j \}_{i=1, j=1}^{n \quad m}$, here $f_{ij}$ is a morphism of $\mathcal{C}$.
\item Compositions of morphisms are defined in the same way as the matrix multiplication.
\end{itemize}

\begin{notation}\label{MatSumNo}
For $\{f_{ij}\} \in Mat (\mathcal{C})$, we often denote it by $\sum_{i, j} f_{ij}$ if there is no danger of confusion.
\end{notation}

Let $Y_0$ and $Y_1$ be compact oriented $0$-manifolds.
For a fixed commutative ring $k$, we extend the category $\mathbf{Cob}_2(Y_0,Y_1)$ to a $k$-linear category $k\mathbf{Cob}_2(Y_0,Y_1)$ with the same objects and $k\mathbf{Cob}_2(Y_0,Y_1)(W_0, W_1)$ being the free $k$-module generated by the hom-set $\mathbf{Cob}_2(Y_0,Y_1)(W_0, W_1)$.
We introduce the following relations on $k\mathbf{Cob}_2(Y_0,Y_1)(W_0,W_1)$ for each cobordisms $W_0$ and $W_1$.
\begin{enumerate}[label=\upshape(\unexpanded{\labelseq{$S$\\$T$\\$4Tu$}}{\value*})]
  \item\label{relK:S} $S \amalg S^2= 0$ for each $2$-bordism $S$, here $S^2$ is the $2$-dimensional sphere;
  \item\label{relK:T} $S\amalg T^2= 2\cdot S$ for each $2$-bordism $S$, here $T^2$ is the $2$-dimensional torus $T^2=S^1\times S^2$;
  \item\label{relK:4Tu} $S_1+S_2-S_3-S_4 = 0$ for each quadruple of $2$-bordisms $S_1$, $S_2$, $S_3$, and $S_4$ which are identical outside disks and tubes that are depicted as follows:
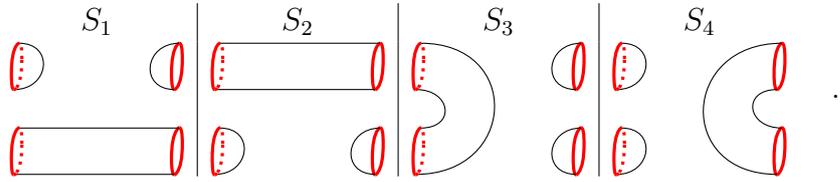
\[
\begin{array}{c|c|c|c}
  S_1 & S_2 & S_3 & S_4 \\
  \BordFourTuL & \BordFourTuR & \BordFourTuU & \BordFourTuD
\end{array}
\quad.
\]
\end{enumerate}
We denote by $k\mathbf{Cob}_2(Y_0,Y_1)/\mathcal L$ the quotient category and set  
\[
 \Cob(Y_0,Y_1)
\coloneqq\operatorname{Mat}(k\mathbf{Cob}_2(Y_0,Y_1)/\mathcal L)
\quad.
\]

\begin{lemma}
\label{lem:K-cobfuncs}
The two functors \eqref{eq:h-composition} and \eqref{eq:cobfunc-disj} induce $k$-bilinear functors
\[
\begin{gathered}
(\blank)\ast(\blank):\Cob(Y_1,Y_2)\times\Cob(Y_0,Y_1)\to\Cob(Y_0,Y_2)
\quad,\\
(\blank)\otimes(\blank):\Cob(Y_0,Y_1)\times\Cob(Y'_0,Y'_1)\to\Cob(Y_0\amalg Y'_0,Y_1\amalg Y'_1)
\quad.
\end{gathered}
\]
In particular, $\Cob(\varnothing,\varnothing)$ is a symmetric monoidal category with $k$-bilinear monoidal product.
\end{lemma}

Similarly, since the relations \ref{relK:S}, \ref{relK:T}, and \ref{relK:4Tu} are stable under orientation reversion, we also have functors below induced by \eqref{eq:cobfunc-rev}:
\begin{equation}
\label{eq:Kfunc-rev}
\begin{gathered}
\rho_0:\Cob(Y_0,Y_1)\to\Cob(\overbar{Y_0},\overbar{Y_1})
\quad,\\
\rho_2:\Cob(Y_0,Y_1)^{\opposite} \to \Cob(Y_0,Y_1)
\quad.
\end{gathered}
\end{equation}

We further extend these functors to complexes in the following way:
let $\mathcal A$, $\mathcal B$, and $\mathcal C$ be $k$-linear categories with $\mathcal C$ being additive.
If $F:\mathcal A\times\mathcal B\to\mathcal C$ is a $k$-bilinear functor, then, for bounded chain complexes $X$ in $\mathcal A$ and $Y$ in $\mathcal B$, we define a chain complex $F(X,Y)$ in $\mathcal C$ by setting
\[
\begin{gathered}
F(X,Y)^n
\coloneqq \bigoplus_{p+q=n}F(X^p,Y^q)
\quad,\\
d^n_{F(X,Y)}\coloneqq \sum_{p+q=n}\left(F(d_X^p,\mathrm{id}_Y)+(-1)^pF(\mathrm{id}_X,d_Y^q)\right)
\quad.
\end{gathered}
\]
We denote by $\mathbf{Ch}^{\mathsf b}(\mathcal A)$ the category of bounded chain complexes and chain maps in $\mathcal A$.
Then, the assignment above yields a $k$-bilinear functor
\[
F:\mathbf{Ch}^{\mathsf b}(\mathcal A)\times\mathbf{Ch}^{\mathsf b}(\mathcal B)\to \mathbf{Ch}^{\mathsf b}(\mathcal C)
\]
which extends the original $F$.
Applying the construction to the functors in \cref{lem:K-cobfuncs}, we in particular obtain functors
\begin{gather}
\label{eq:Kom-fun:hcomp}
(\blank)\ast(\blank):\mathbf{Ch}^{\mathsf b}(\Cob(Y_1,Y_2))\times\mathbf{Ch}^{\mathsf b}(\Cob(Y_0,Y_1))\to\mathbf{Ch}^{\mathsf b}(\Cob(Y_0,Y_2))
\quad,\\
\label{eq:Kom-fun:disj}
(\blank)\otimes(\blank):\mathbf{Ch}^{\mathsf b}(\Cob(Y_0,Y_1))\times\mathbf{Ch}^{\mathsf b}(\Cob(Y'_0,Y'_1))\to \mathbf{Ch}^{\mathsf b}(\Cob(Y_0\amalg Y'_0,Y_1\amalg Y'_1))
\quad.
\end{gather}
We also extend the functors $\rho_0$ and $\rho_2$ in \eqref{eq:Kfunc-rev} by
\begin{gather}
\label{eq:chain-rho0}
\begin{array}[t]{rccc}
  \rho_0:&\mathbf{Ch}^{\mathsf b}(\Cob(Y_0,Y_1))&\to&\mathbf{Ch}^{\mathsf b}(\Cob(\overbar{Y_0},\overbar{Y_1})) \\
       & \{X^i,d^i\} &\mapsto& \{\rho_0(X^i),\rho_0(d^i)\}_i
\end{array}
\quad,\\
\label{eq:chain-rho2}
\begin{array}[t]{rccc}
  \rho_2:&\mathbf{Ch}^{\mathsf b}(\Cob(Y_0,Y_1))^{\opposite}&\mapsto&\mathbf{Ch}^{\mathsf b}(\Cob(Y_0,Y_1)) \\
       & \{X^i,d^i\} &\mapsto& \{\rho_2(X^{-i}),\rho_2(d^{-i-1})\}_i
\end{array}
\quad.
\end{gather}

\subsection{The universal bracket complex}\label{sec:univ-Kh:univBracket}

In this section, we construct the   universal bracket complex of tangle diagrams.  To begin with, we define the modules of signs.

Let $\mathcal{S}$ be the totally ordered set.
For each subset $A \subset {\mathcal{S}}$, we set $E_A:=\varnothing\in\Cob(\varnothing,\varnothing)$, which is the unit in the monoidal structure.
For each $a\in\mathcal{S}$, we define the morphisms $\cra{a}$, $( \wedge a)$, and $(a \wedge )$ as follows.
Let $\mu_a=\#\{a'\in A\mid a'<a\}$ and $\nu_a=\#\{a'\in A\mid a'>a\}$.
Then, we set
\begin{gather}
\label{SymbolCheck}
\cra{a}: E_A \to E_{A \setminus \{ a \} }:=
\begin{cases*}
(-1)^{\mu_a} & if $a \in A$,\\
0 & if $a \notin A$,
\end{cases*}
\\
\label{SymbolLeftWedge}
(\wedge a): E_A \to E_{A \cup \{ a \} }:=
\begin{cases*}
(-1)^{\mu_a} & if $a \notin A$,\\
0 & if $a \in A$,
\end{cases*}
\\
\label{SymbolRightWedge}
(a\wedge):E_A \to E_{A \cup \{ a \} } :=
\begin{cases*}
(-1)^{\nu_a} & if $a \notin A$,\\
0 & if $a \in A$.
\end{cases*}
\end{gather}

\begin{notation}
We often denote $(a \wedge )$ by $a_{\dagger}$.
\end{notation}  

Let $D$ be a tangle diagram, which we regard as a planar graph with boundary neatly embedded in $\mathbb R\times[0,1]$.
We denote by $c(D)$ the set of crossings in $D$ and call each subset $s\subset c(D)$ a \emph{state} on $D$; we write $|s|$ the cardinality.
For each state $s$ on $D$, we write $D_s$ the compact $1$-dimensional neat submanifold of $\mathbb R\times[0,1]$ obtained by smoothing each crossing of $D$ according to $s$:
\[
\diagSmoothH
\;\xleftarrow[\text{$0$-smoothing}]{\displaystyle c\notin s}\;
\diagCrossNegWith{c}
\;\xrightarrow[\text{$1$-smoothing}]{\displaystyle c\in s}\;
\diagSmoothV
\quad.
\]
Hence, each $D_s$ is a neat submanifold of $\mathbb R\times[0,1]$.
We endow $D_s$ with an orientation as follows: recall that a \emph{checkerboard coloring} on the complement $(\mathbb R\times[0,1])\setminus D$ is a mapping
\[
\chi:\pi_0((\mathbb R\times[0,1])\setminus D)\to \{\text{white},\text{black}\}
\]
which distinguishes adjacent components.
If a checkerboard coloring $\chi$ on the complement of $D$ is fixed, it induces a checkerboard coloring on $(\mathbb R\times[0,1])\setminus D_s $ for each state $s$  which we also write $\chi$ by abuse of notation.
Then, we denote by $D_s^\chi$ the manifold $D_s$ equipped with the canonical orientation on the boundary of the black component with respect to $\chi$; i.e. $D_s^\chi= \partial(\chi^{-1}\{\text{black}\})$ as oriented manifolds.
Note that there are exactly two checkerboard colorings.
Namely, if $\chi$ is a checkerboard coloring on $(\mathbb R\times[0,1])\setminus D$, then the other is obtained by swapping all the values of $\chi$, which we denote by $-\chi$.
In this case, the $1$-manifold $D_s^{-\chi}$ is identified with $D_s^\chi$ with the reversed orientation.

Since the induced orientation on the boundary $\partial D_s^\chi$ does not depend on the state $s$, we in particular write $\partial D^\chi\coloneqq\partial D_\varnothing^\chi$ and
\[
\partial^- D^\chi\coloneqq \overbar{\partial D^\chi\cap(\mathbb R\times\{0\})}
\ ,\quad \partial^+ D^\chi\coloneqq \partial D^\chi\cap(\mathbb R\times\{1\})
\quad.
\]
In fact, the orientations on them are determined locally by the rules  as in \cref{fig:D-bndry-ori}.
\begin{figure}[tbp]
\centering
\begin{tikzpicture}
\node[left] at (0,0) {$\mathbb R\times\{1\}$};
\node[above=1ex,blue] at (3.5,0) {$\partial^+D^\chi$};
\fill[lightgray] (2,0) rectangle (.5,-1);
\fill[lightgray] (5,0) rectangle (6.5,-1);
\draw (0,0) -- (7,0);
\draw[red,very thick] (2,0) -- (2,-1);
\draw[red,very thick] (5,0) -- (5,-1);
\fill[blue] (2,0) circle(.15) node[above]{$+$};
\fill[blue] (5,0) circle(.15) node[above]{$-$};
\end{tikzpicture}
\qquad
\begin{tikzpicture}
\node[left] at (0,0) {$\mathbb R\times\{0\}$};
\node[below=1ex,blue] at (3.5,0) {$\partial^-D^\chi$};
\fill[lightgray] (2,0) rectangle (.5,1);
\fill[lightgray] (5,0) rectangle (6.5,1);
\draw (0,0) -- (7,0);
\draw[red,very thick] (2,0) -- (2,1);
\draw[red,very thick] (5,0) -- (5,1);
\fill[blue] (2,0) circle(.15) node[below]{$+$};
\fill[blue] (5,0) circle(.15) node[below]{$-$};
\end{tikzpicture}
\caption{The orientation on $\partial^-D^\chi$ and $\partial^+D^\chi$.}
\label{fig:D-bndry-ori}
\end{figure}
Thus, for each state $s$ on $D$, we may regard $D_s^\chi$ as an object of the category $\mathbf{Cob}_2(\partial^-D^\chi,\partial^+D^\chi)$ and hence of $\Cob(\partial^-D^\chi,\partial^+D^\chi)$.

For a tangle diagram $D$ with a checkerboard coloring $\chi$ on $(\mathbb R\times[0,1])\setminus D$, we define a graded $k$-module $\brcOf*{D^\chi}$ by
\[
\brcOf*{D^\chi}^i
\coloneqq \bigoplus_{s\subset c(D),\,|s|=i} D_s^\chi\otimes E_s
\in \Cob(\partial^-D^\chi,\partial^+D^\chi)
\]
for each integer $i\in\mathbb Z$.
We in addition endow $\brcOf*{D^\chi}$ with a differential as follows:
for each pair $(s,c)$ of a state $s\subset c(D)$ and a crossing $c\in c(D)$ with $c\notin s$, notice that, $D_s$ and $D_{s\cup\{c\}}$ is identical except on a neighborhood of the crossing $c$ where they are of the following forms regardless of the orientation:
\[
\begin{array}{c|c|c}
  D & D_s& D_{s\cup\{c\}} \\\hline\rule{0pt}{5ex}
  \diagCrossNegWith{c}\; & \;\diagSmoothH\; & \;\diagSmoothV
\end{array}
\quad.
\]
We define a cobordism $S_{s;c}:D_s\to D_{s\cup\{c\}}$ by
\begin{equation}
\label{eq:saddle-cob}
\BordDeltaN
\;:\;
\diagSmoothH
\;\rightarrow\; \diagSmoothV
\end{equation}
on the neighborhood and the identity elsewhere.
Thanks to the stability of checkerboard colorings under smoothing, $S_{s;c}$ has an obvious orientation which make $S_{s;c}$ as an oriented cobordism $D_s^\chi\to D_{s\cup\{c\}}^\chi$.
Hence, we obtain a morphism
\[
S_{s;c}\otimes(\wedge c):D_s^\chi\otimes E_s\to D_{s\cup\{c\}}^\chi\otimes E_{s\cup\{c\}}
\quad.
\]
We then define the differential by
\begin{equation}
\label{eq:brc-diff}
d^i\coloneqq \sum_{s\subset c(D),\,|s|=i,\,c\in c(D)\setminus s} S_{s;c}\otimes(\wedge c):\brcOf*{D^\chi}^i\to \brcOf*{D^\chi}^{i+1}
\quad.
\end{equation}
We call $\brcOf*{D^\chi}$ the \emph{universal bracket complex} of $D$.

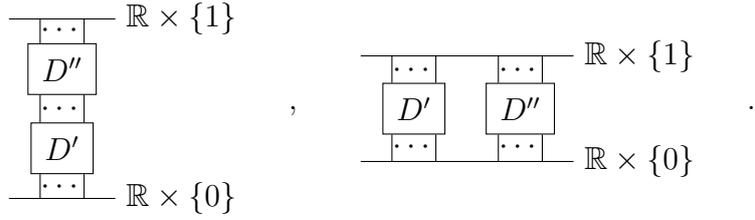
\begin{figure}[t]
\centering
\begin{tikzpicture}[baseline=(current bounding box.center)]
\node[rectangle,draw,inner sep=1ex] (DU) at (0,.75) {$D''$};
\node[rectangle,draw,inner sep=1ex] (DD) at (0,-.75) {$D'$};
\node at (0,0) {$\cdots$};
\node[below=1.5ex] at (DD) {$\cdots$};
\node[above=1.5ex] at (DU) {$\cdots$};
\draw (-1,1.7) -- (1,1.7) node[right]{$\mathbb R\times\{1\}$};
\draw (-1,-1.7) -- (1,-1.7) node[right]{$\mathbb R\times\{0\}$};
\draw[transform canvas={xshift=-.7em}] (0,-1.7) -- (DD) -- (DU) -- (0,1.7);
\draw[transform canvas={xshift=.7em}] (0,-1.7) -- (DD) -- (DU) -- (0,1.7);
\end{tikzpicture}
\quad,\qquad
\begin{tikzpicture}[baseline=(current bounding box.center)]
\node[draw,rectangle,inner sep=1ex] (DL) at (-1,0) {$D'$};
\node[draw,rectangle,inner sep=1ex] (DR) at (1,0) {$D''$};
\node[above=1.5ex] at (DL) {$\cdots$};
\node[below=1.5ex] at (DL) {$\cdots$};
\node[above=1.5ex] at (DR) {$\cdots$};
\node[below=1.5ex] at (DR) {$\cdots$};
\draw (-2,1) -- (2,1) node[right]{$\mathbb R\times\{1\}$};
\draw (-2,-1) -- (2,-1) node[right]{$\mathbb R\times\{0\}$};
\draw[transform canvas={xshift=-.7em}] (-1,-1) -- (DL) -- (-1,1);
\draw[transform canvas={xshift=.7em}] (-1,-1) -- (DL) -- (-1,1);
\draw[transform canvas={xshift=-.7em}] (1,-1) -- (DR) -- (1,1);
\draw[transform canvas={xshift=.7em}] (1,-1) -- (DR) -- (1,1);
\end{tikzpicture}
\quad.
\caption{The composition (left) and the tenser product (right) of  tangle diagrams.}\label{fig:Comp_tangle}
\end{figure}

The following results show that the universal bracket complex respects the operation on tangles in terms of the functors (\ref{eq:Kom-fun:hcomp}) and (\ref{eq:Kom-fun:disj}). 

\begin{proposition}[{\cite[Theorem~2]{BarNatan2005}}]
\label{prop:TangleAst}
Let $D$ be the composition of two tangle diagrams $D'$ and $D''$ as in \cref{fig:Comp_tangle}.
For a checkerboard coloring $\chi$ on $(\mathbb R\times[0,1])\setminus D$, let us write $\chi'$ and $\chi''$ respectively the induced coloring on the complements of $D'$ and $D''$.
Then, there is an isomorphism
\begin{equation}
\label{eq:UKH-tang-comp}
\brcOf*{D^\chi}\cong\brcOf*{D'^{\chi'}}\ast\brcOf*{D''^{\chi''}}
\end{equation}
in the category $\mathbf{Ch}^{\mathsf b}(\Cob(\partial^-D^\chi,\partial^+D^\chi))$.  
\end{proposition}

\begin{proposition}[{\cite[Theorem~2]{BarNatan2005}}]
\label{prop:TangleOtimes}
Let $D$ be the tensor product of tangle diagrams $D'$ and $D''$ as in \cref{fig:Comp_tangle}.
For a checkerboard coloring $\chi$ on $(\mathbb R\times[0,1])\setminus D$, we write $\chi'$ and $\chi''$ respectively the induced coloring on the complements of $D'$ and $D''$.
Then, there is an isomorphism
\[
\brcOf*{D^\chi}
\cong \brcOf*{D'^{\chi'}}\otimes\brcOf*{D''^{\chi''}}
\]
in the category $\mathbf{Ch}^{\mathsf b}(\Cob(\partial^-D^\chi,\partial^+D^\chi))$.
\end{proposition}

\subsection{The universal Khovanov complex}\label{sec:univ-Kh:univKh}

We now introduce the complex $\KhOf*{D}$ that is an invariant of tangles.
We always assume a tangle $T$ to be ``generic'' so that the image of $T$ under the projection $\mathbb R^2\times[0,1]\to\mathbb R\times[0,1]$ defines a tangle diagram $D$; in this case, we call $D$ \emph{the} diagram of $T$.
We say that two tangles are \emph{isotopic} if they are connected by an ambient smooth isotopy which is the identity on the boundary.
A connected component of $(\mathbb R\times[0,1])\setminus D$ is said to be \emph{negatively unbounded} if it contains the point $(-x,\frac12)$ for arbitrarily large $x>0$.

\begin{notation}\label{NotationShift}
If $W=\{W^{i}, d^i\}_{i}$ is a chain complex, then we define a chain complex $W[k]$ by
\[
W[k]^{i}
\coloneqq W^{i-k}
\ ,\quad
d^i_{W[k]} \coloneqq (-1)^k d^i
\quad.
\] 
\end{notation}

\begin{definition}
Let $D$ be a tangle diagram with  $n_-$ negative crossings.
Let $\chi_w$ be the checkerboard coloring with negatively unbounded white component.
Then, we set
\begin{equation}
\label{eq:univKh-def}
\KhOf*{D}\coloneqq \brcOf*{D^{\chi_w}}[-n_-]
\in\Cob(\partial^-D^{\chi_w},\partial^+D^{\chi_w})
\end{equation}
and call it the \emph{universal Khovanov complex} of $D$.
\end{definition}

If $D$ and $D'$ are the diagrams of two isotopic tangles, then the restriction on the isotopies guarantees that $\partial D=D\cap(\mathbb R\times\{0,1\})$ and $\partial D'=D'\cap(\mathbb R\times\{0,1\})$ are mutually identical.
This in particular implies that we have
\[
\partial D^{\chi}=\partial {D'}^{\chi'}
\]
as oriented $0$-manifolds provided $\chi$ and $\chi'$ have the same color at the negatively unbounded components.
It follows that $\KhOf*{D}$ and $\KhOf*{D'}$ lie in the same category.

\begin{theorem}[{\cite[Theorem~1]{BarNatan2005}}]\label{thm:Bar-Natan}
The homotopy type of $\KhOf*{D}^i$ is an isotopy invariant of tangles.
\end{theorem}

\subsection{The universal bracket complex as a mapping cone}

\begin{definition}[mapping cone]\label{DefMappingCone}
Let $\mathcal A$ be an additive  category.
If $f:X\to Y$ is a chain map  between chain complexes in $\mathcal A$, then the \emph{mapping cone} $\operatorname{Cone}(f)$ is a chain complex defined as follows:
\begin{itemize}
  \item as an object of $\mathcal A$, we have
\[
\operatorname{Cone}(f)^i = Y^i\oplus X^{i+1}\ ;
\]
  \item the differential $d^i=d^i_{\operatorname{Cone}(f)}:\operatorname{Cone}(f)^i\to\operatorname{Cone}(f)^{i+1}$ is presented by the matrix
\[
d^i_{\operatorname{Cone}(f)}\coloneqq
\begin{pmatrix}
d^i_Y & f \\
0 & -d^{i+1}_X
\end{pmatrix}
:Y^i\oplus X^{i+1}\to Y^{i+1}\oplus X^{i+2}
\quad.
\]
\end{itemize}
\end{definition}

Since $\operatorname{Cone}(f)$ is actually a chain complex, we call it  the \emph{mapping cone} of $f$.

For a tangle diagram $D$, fix a crossing $c \in c(D)$, and set $D^{(0)}$ and $D^{(1)}$ the diagrams obtained from $D$ by applying $0$- and $1$-smoothing to $c$ respectively.
We hence have a canonical identification $c(D^{(0)})=c(D^{(1)})=c(D)\setminus\{c\}$.
Then, the saddle cobordism induces the morphism
\[
\delta_c \coloneqq \BordDeltaN
\;:\;
\brcOf*{{D^{(0)}}^{\chi}} \to \brcOf*{{D^{(1)}}^{\chi}}
\quad.
\]

\begin{proposition}\label{prop:coneVSbracket}
In the situation above, there is an isomorphism
\[
\brcOf*{D}\cong\operatorname{Cone}(-\delta_c)[1]
\quad.
\]
\end{proposition}
\begin{proof}
Note that, for each $s\subset c(D)\setminus\{c\}$, there are identifications
\[
(D^{(0)}_s)^\chi = D^\chi_s
\ ,\quad (D^{(1)}_s)^\chi = D^\chi_{s\amalg\{c\}}
\quad.
\]
We hence define a morphism $\brcOf{(D^{(1)})^\chi}\oplus\brcOf{(D^{(0)})^\chi}\to\brcOf{D^\chi}$ consisting of
\[
\mathrm{id}\otimes(\wedge c):(D^{(1)}_s)^\chi\otimes E_s \to D^\chi_{s\amalg\{c\}}\otimes E_{s\amalg\{c\}}
\ ,\quad
\mathrm{id}\otimes\mathrm{id}:(D^{(0)}_s)^\chi\otimes E_s\to D^\chi_s\otimes E_s
\quad.
\]
By comparing the differentials, one can easily verify that this is actually an isomorphism of chain complexes.
\end{proof}

\subsection{Duality with respect to mirroring}
\label{sec:univ-Kh:dual}

To conclude the section, we see the dualities of the universal Khovanov complex in terms of functors \eqref{eq:chain-rho0} and \eqref{eq:chain-rho2}.
In order to establish them, we need some technical materials on the modules of signs.
Let $\mathcal S$ be a finite totally ordered set, say $n=|\mathcal S|$.
For a subset $A\subset\mathcal S$, we write $\varepsilon_A$ the sign of the $(|A|,n-|A|)$-shuffle and think of it as a morphism
\[
\varepsilon_A:E_A\to E_{\mathcal S\setminus A}
\in \Cob(\varnothing,\varnothing)
\quad.
\]
It then turns out that the diagram below commutes:
\begin{equation}
\label{eq:eps-dual}
\begin{tikzcd}[column sep=4em]
E_{A\amalg\{c\}} \ar[r,"\cra{c}"] \ar[d,"\varepsilon_{A\amalg\{c\}}"'] & E_A \ar[d,"\varepsilon_A"] \\
E_{(\mathcal S\setminus A)\setminus\{c\}} \ar[r,"(-1)^{n-1}(\wedge c)"] & E_{\mathcal S\setminus A}
\end{tikzcd}
\quad.
\end{equation}
In terms of the universal bracket complex, the dualities are stated as follows.

\begin{proposition}\label{prop:bracket-dual}
Let $D^{\mathsf{mir}}$ be the mirror image of a tangle diagram $D$ with $n$ crossings.
Then, for every checkerboard  coloring $\chi$, there are isomorphisms
\[
\rho_0 (\brcOf*{D^{\chi}}) \cong \brcOf*{D^{-\chi}}
\ ,\quad
\rho_2 (\brcOf*{D^{\chi}})  \cong \brcOf*{(D^{\mathsf{mir}})^\chi}[-n]
\quad.
\]
\end{proposition}
\begin{proof}
Since the first isomorphism is obvious, we prove the second.
We identify the set $c(D^{\mathsf{mir}})$ of crossings in $D^{\mathsf{mir}}$ with $c(D)$.
Hence, for each state $s\subset c(D)$, there is a canonical identification $D_s^\chi=(D^{\mathsf{mir}}_{\overbar s})^\chi$ with $\overbar s\coloneqq c(D)\setminus s$.
We set
\begin{equation}
\label{eq:prf:bracket-dual:iota}
\iota_s\coloneqq (-1)^{|s|}\mathrm{id}\otimes\varepsilon_s: D_s^\chi\otimes E_s\to (D^{\mathsf{mir}}_{\overbar s})^\chi\otimes E_{\overbar s}
\end{equation}
and write $\iota^i:\brcOf{D^\chi}^{-i}\to\brcOf{(D^{\mathsf{mir}})^\chi}^{n+i}$ the induced morphism.
We assert that the family $\iota=\{\iota^i\}_i$ defines a morphism of chain complexes $\rho_2(\brcOf{D^\chi})\to \brcOf{(D^{\mathsf{mir}})^\chi}[-n]$.
Indeed, for each $s\subset c(D)$ and $c\in c(D)\setminus s$, since we have $\rho_2((c\wedge))=\cra{c}:E_{s\amalg c}\to E_s$, the square \eqref{eq:eps-dual} yields a commutative square
\begin{equation}
\label{eq:prf:bracket-dual:diffcomp}
\begin{tikzcd}[column sep=6em]
D_{s\amalg\{c\}}^\chi\otimes E_{s\amalg\{c\}} \ar[r,"(-1)^{|s|}\rho_2(S_{s;c}^D\otimes(\wedge c))"] \ar[d,"\iota_{s\amalg\{c\}}"] & D_s^\chi\otimes E_s \ar[d,"\iota_s"] \\
(D_{\overbar s\setminus\{c\}})^\chi\otimes E_{\overbar s\setminus\{c\}} \ar[r,"(-1)^nS_{\overbar s\setminus\{c\};c}^{D^{\mathsf{mir}}}\otimes(\wedge c)"] & (D_{\overbar s}^{\mathsf{mir}})^\chi\otimes E_{\overbar s}
\end{tikzcd}
\quad,
\end{equation}
here $S^D_\ast$ and $S^{D^{\mathsf{mir}}}_\ast$ are the saddle cobordisms \eqref{eq:saddle-cob} which appear in the differentials.
This implies that $\iota$ is a morphism of chain complexes.
Since it is obviously an isomorphism, this completes the proof.
\end{proof}

\begin{corollary}
\label{cor:UKH-dual}
Let $D^{\mathsf{mir}}$ be the mirror image of a tangle diagram $D$.
Then, there is an isomorphism
\[
\rho_2(\KhOf{D})\cong\KhOf{D^{\mathsf{mir}}}
\quad.
\]
\end{corollary}

\section{Genus-one morphism}\label{sec:genus-1}

We now define a morphism of chain complexes   
\begin{equation}
\label{eq:map-PhiHat}
\widehat\Phi:
 \brcOf*{\diagCrossNegUp^{\chi}}
\to \brcOf*{\diagCrossPosUp^{\chi}}[1]
\quad.
\end{equation}

\begin{lemma}[{\cite[Proposition~3.1.3]{Verdier1996}}]
\label{lem:cone=cofib}
Suppose we have a sequence
\[
X\xrightarrow{f} Y\xrightarrow{g} Z
\]
of chain morphisms in an additive  category $\mathcal A$.
If there is a chain homotopy $H:gf\Rightarrow 0$, that is, $dH + Hd = -gf$, then the morphism $g$ factors through a morphism $\widehat g:\operatorname{Cone}(f)\to Z$ given by
\begin{equation}
\label{eq:cone-nullinduce}
\widehat g^i=
\begin{pmatrix}
g & -H
\end{pmatrix}
:\operatorname{Cone}(f)^i=Y^i\oplus X^{i+1}\to Z^i
\quad
\end{equation}
following the canonical morphism  $Y\to\operatorname{Cone}(f)$.
\end{lemma}

We define the morphism $\Phi$ on the universal bracket complex $\brcOf{\blank}$  induced by the following cobordism:
\begin{equation}\label{eq:Phi}
\BordPhiFst - \BordPhiSnd
\;:\;
\diagSmoothV\to \diagSmoothV
\quad.
\end{equation}
We also have the following single saddle operations:
\[
\BordDeltaN
\;:\;
\diagSmoothH\to \diagSmoothV
\quad,\qquad
\BordDeltaP
\;:\;
\diagSmoothV
\to \diagSmoothH
\quad.
\]
We denote by $\delta_-$ and $\delta_+$ respectively the morphism induced on complexes $\brcOf{\blank}$.
We obtain the sequence of morphisms of chain complexes below:
\begin{equation}
\label{eq:delta-Phi-delta}
\brcOf*{
{\diagSmoothH}^{\chi} }
\xrightarrow{-\delta_-}
\brcOf*{ {\diagSmoothUp}^{\chi} }
\xrightarrow{\Phi}
\brcOf*{{\diagSmoothUp}^{\chi} }
\xrightarrow{-\delta_+ } \brcOf*{{\diagSmoothH}^{\chi} }\quad.
\end{equation}

\begin{proposition}
\label{prop:delta-Phi}
In the situation above, the compositions $\Phi\delta_-$ and $\delta_+\Phi$ are zero. 
Consequently, the sequence \eqref{eq:delta-Phi-delta} induces a morphism of chain complexes
\[
\widehat\Phi = \BordPhiFst \otimes \cra{c} - \BordPhiSnd \otimes \cra{c}:\;
\brcOf*{{\diagCrossNegUpWith{c}}^{\chi}}
\to
\brcOf*{{\diagCrossPosUpWith{c}}^{\chi}}[1].
\]
\end{proposition}
\begin{proof}
The first statement follows from the equations: 
\[
\BordPhiFst \BordDeltaP  
=
\BordPhiSnd \BordDeltaP 
\quad, \qquad
\BordDeltaN \BordPhiFst   
=
\BordDeltaN \BordPhiSnd  
\quad. 
\]

We show the latter.
By Proposition~\ref{prop:coneVSbracket}, we have identifications
\[
\operatorname{Cone}(-\delta_-) 
\cong  \brcOf*{ \diagCrossNegUp^{\chi} }[-1]
\ ,\quad
\operatorname{Cone}(-\delta_+) 
\cong  \brcOf*{ \diagCrossPosUp^{\chi} }[-1]
\ .
\]
Hence, in view of Lemma~\ref{lem:cone=cofib}, \cref{prop:delta-Phi} yields a morphism of chain complexes $\widehat\Phi$ as required.
\end{proof}

In what follows, the morphism $\widehat\Phi$ is referred to as the \emph{genus-one morphism}.
The morphism $\widehat\Phi$ induces a morphism
\[
\KhOf*{\diagCrossNegUp}
\to \KhOf*{\diagCrossPosUp}
\quad.
\]
Moreover, it is of degree $0$ with respect to Euler graded TQFT \cite{BarNatan2005,LaudaPfeiffer2009}.

\begin{remark}
\label{rem:Phi-prime}
In the definition of the morphism $\Phi$, for the position of the $1$-handle attaching in the second term, if we switch ``left'' to ``right'' and define $\Phi'$, we have $\Phi'=-\Phi$ thanks to the relation \ref{relK:4Tu} in \cref{sec:univ-Kh:K}.
\end{remark}

\begin{proposition}\label{prop:Mir}
Let $D_-$ and $D_+$ be the same tangle diagram except for a crossing $c$ whose sign is negative and positive, respectively.
Let ${D}^{\mathsf{mir}}_-$ and ${D}^{\mathsf{mir}}_+$ be the mirror images of ${D}_-$ and ${D}_+$, respectively.
Let $\chi$ be a checkerboard coloring.
The crossing in $D^{\mathsf{mir}}_{\pm}$ corresponding to $c$ is denoted by $c^{\mathsf{mir}}$.   Let $\widehat{\Phi}_c :$ $\brcOf*{D_-^{\chi}}$ $\to$ $\brcOf*{D_+^{\chi}}$ be the genus-one morphism that is applied to $c$.
Then the following diagram commutes:
\begin{equation}
\label{diag:Mirror}
\begin{tikzcd}
\rho_2 \left( \brcOf*{D_+^{\chi}} \right) \ar[r,"\rho_2 (\widehat{\Phi}_c)"] \ar[d,"\cong"'] & \rho_2 \left( \brcOf*{D_-^{\chi}} \right) \ar[d,"\cong"]
\\
\brcOf*{{D_+^{\mathsf{mir}}}^{\chi}}[-n] \ar[r,"\widehat{\Phi}_{c^{\mathsf{mir}}}"'] & \brcOf*{{D_-^{\mathsf{mir}}}^{\chi}}[-n] 
\end{tikzcd}
\quad,
\end{equation}
here the vertical isomorphisms are the ones in \cref{prop:bracket-dual}.
\end{proposition}

\Cref{prop:Mir} is verified by the direct computation, so we omit the proof.

\begin{remark}
We note that, though there are other choices on morphisms of chain complexes of the form \eqref{eq:map-PhiHat}, some popular ones fail to have of bidegree $(0,0)$ in the case of Khovanov homology.
For example, the following link cobordisms realize another crossing change:
\begin{equation}
\label{eq:intro:cob-crossingchange}
\diagCrossNegUp{}=
\begin{tikzpicture}[baseline=-.5ex]
\node[circle,inner sep=2] (L) at (-.2,0) {};
\node[circle,inner sep=2] (R) at (.2,0) {};
\draw[red,very thick,-stealth] (-.6,-.8) -- (L) (L) to[out=56,in=180] (R.north) to[out=0,in=236] (.6,.8);
\draw[red,very thick,-stealth] (.6,-.8) to[out=124,in=0] (R.south) to[out=180,in=-56] (L.center) -- (-.6,.8);
\end{tikzpicture}
\xrightarrow{\text{saddle}}
\begin{tikzpicture}[baseline=-.5ex]
\node[circle,inner sep=2] (L) at (-.2,0) {};
\node[circle,inner sep=1] (R) at (.2,0) {};
\draw[red,very thick,-stealth] (-.6,-.8) -- (L) (L) to[out=56,in=90,looseness=3] (R.west) to[out=-90,in=-56,looseness=3] (L.center) -- (-.6,.8);
\draw[red,very thick,-stealth] (.6,-.8) to[out=124,in=-90] (R.east) to[out=90,in=236] (.6,.8);
\end{tikzpicture}
\xrightarrow{\RMove1^2}
\begin{tikzpicture}[baseline=-.5ex]
\node[circle,inner sep=2] (L) at (-.2,0) {};
\node[circle,inner sep=1] (R) at (.2,0) {};
\draw[red,very thick,-stealth] (-.6,-.8) -- (L.center) to[out=56,in=90,looseness=3] (R.west) to[out=-90,in=-56,looseness=3] (L) (L) -- (-.6,.8);
\draw[red,very thick,-stealth] (.6,-.8) to[out=124,in=-90] (R.east) to[out=90,in=236] (.6,.8);
\end{tikzpicture}
\xrightarrow{\text{saddle}}
\begin{tikzpicture}[baseline=-.5ex]
\node[circle,inner sep=2] (L) at (-.2,0) {};
\node[circle,inner sep=2] (R) at (.2,0) {};
\draw[red,very thick,-stealth] (-.6,-.8) -- (L.center) to[out=56,in=180] (R.north) to[out=0,in=236] (.6,.8);
\draw[red,very thick,-stealth] (.6,-.8) to[out=124,in=0] (R.south) to[out=180,in=-56] (L) (L) -- (-.6,.8);
\end{tikzpicture}
=\diagCrossPosUp
\quad.
\end{equation}
It turns out that the induced morphism on Khovanov complexes coincides with the following composition:
\[
\brcOf*{\diagCrossPos^{\chi}}
\xrightarrow{\beta} \brcOf*{\diagSmoothV^{\chi}}
\xrightarrow\alpha  \brcOf*{\diagCrossNeg^{\chi}}[-1]
\quad.
\]
M.~Hedden and L.~Watson \cite[Section~3.1]{HeddenWatson2018} used this morphism to derive a categorified version of the Jones skein relation.
\end{remark}

\section{Invariance}
\label{sec:invariance}

In this section, we see that the genus-one morphism $\widehat\Phi$ defined in \cref{prop:delta-Phi} is invariant under moves involved with singular links.
Namely, according to \cite{BatainehElhamdadiHajijYoumans2018}, two singular link diagrams represent isotopic singular links if and only if they are connected by the following moves in addition to Reidemeister moves:
\begin{equation}
\label{eq:moveSing}
\diagCrossSingRivOL \leftrightarrow \diagCrossSingRivOR
\ ,\quad
\diagCrossSingRivUL \leftrightarrow \diagCrossSingRivUR
\ ,\quad
\diagRvSingU \leftrightarrow \diagRvSingD
\quad.
\end{equation}
Motivated by this fact, we aim at proving the invariance of $\widehat\Phi$ under these moves in the sense of the following propositions, where the checkerboard colorings are omitted from the notation for simplicity.

\begin{theorem}
\label{theo:genus1-inv}
There are chain-homotopy commutative squares
\begin{gather}
\label{eq:genus1-inv:O4}
\begin{tikzcd}[column sep=1.5em,row sep=3ex,ampersand replacement=\&]
\brcOf*{\tikz[baseline=-.5ex]{\nodeTikzBox{diagCrossNegRivOL}; \node at (.5,0) {$c_-$};}} \ar[r,"\widehat\Phi_c"] \ar[d,"\simeq"'] \& \brcOf*{\tikz[baseline=-.5ex]{\nodeTikzBox{diagCrossPosRivOL}; \node at (.5,0){$c_+$};}}[1] \ar[d,"\simeq"] \ar[dl,Rightarrow] \\
\brcOf*{\tikz[baseline=-.5ex]{\nodeTikzBox{diagCrossNegRivOR}; \node at (-.5,0) {$c'_-$};}} \ar[r,"\widehat\Phi_{c'}"] \& \brcOf*{\tikz[baseline=-.5ex]{\nodeTikzBox{diagCrossPosRivOR}; \node at (-.5,0) {$c'_+$};}}[1]
\end{tikzcd}
,\,
\begin{tikzcd}[column sep=1.5em,row sep=3ex,ampersand replacement=\&]
\brcOf*{\tikz[baseline=-.5ex]{\nodeTikzBox{diagCrossNegRivUL}; \node at (.5,0) {$c_-$};}} \ar[r,"\widehat\Phi_c"] \ar[d,"\simeq"'] \& \brcOf*{\tikz[baseline=-.5ex]{\nodeTikzBox{diagCrossPosRivUL}; \node at (.5,0){$c_+$};}}[1] \ar[d,"\simeq"] \ar[dl,Rightarrow] \\
\brcOf*{\tikz[baseline=-.5ex]{\nodeTikzBox{diagCrossNegRivUR}; \node at (-.5,0) {$c'_-$};}} \ar[r,"\widehat\Phi_{c'}"] \& \brcOf*{\tikz[baseline=-.5ex]{\nodeTikzBox{diagCrossPosRivUR}; \node at (-.5,0) {$c'_+$};}}[1]
\end{tikzcd}
,\displaybreak[2]\\[1ex]
\label{eq:genus1-inv:O5}
\begin{tikzcd}[column sep=1.5em,row sep=3ex,ampersand replacement=\&]
\brcOf*{\diagRiiLeftUpWithSign} \ar[r,"\widehat\Phi_a"] \ar[d,"\simeq"'] \& \brcOf*{\diagRvFTwWithSign}[1] \ar[d,equal] \ar[dl,Rightarrow] \\
\brcOf*{\diagRiiRightUpWithSign} \ar[r,"\widehat\Phi_b"] \& \brcOf*{\diagRvFTwWithSign}[1]
\end{tikzcd}
\quad,
\end{gather}
with vertical edges being chain homotopy equivalences.
\end{theorem}

\begin{remark}\label{rem:invO4e}
We note that the second move in \eqref{eq:moveSing} is the mirror image of the first.
By virtue of the duality of Khovanov homology (cf.~\cref{prop:bracket-dual} and \cref{prop:Mir}), this implies that invariance under one move follows from that under the other.
In the proof of \cref{theo:genus1-inv}, we concentrate on the first move in particular.
\end{remark}

\subsection{Homotopy coherence of mapping cones}
\label{sec:invariance:cone}

As the genus-one morphism $\widehat\Phi$ is obtained from the sequence \eqref{eq:delta-Phi-delta}, the invariance stated in \cref{theo:genus1-inv} will be inherited from that of \eqref{eq:delta-Phi-delta}.
We then begin with a discussion on this kind of inheritance of invariance.

We first see that mapping cones have functoriality with respect not only to commutative squares but also to chain-homotopy commutative ones.
Let $\mathcal A$ be an additive category, and suppose we are given a chain-homotopy commutative diagram
\begin{equation}
\label{eq:hmtp-commsq}
\begin{tikzcd}
X' \ar[r,"f'"] \ar[d,"u"'] & Y' \ar[d,"v"] \ar[dl,Rightarrow,"F"'] \\
X \ar[r,"f"'] & Y
\end{tikzcd}
\quad;
\end{equation}
in other words, $F$ is a chain homotopy with $d_YF+Fd_{X'}=fu-vf'$.
We define a morphism
\begin{equation}
\label{eq:cone-hinduced}
F^i_\ast\coloneqq
\begin{pmatrix}
v^i & -F^i \\ 0 & u^{i+1}
\end{pmatrix}
:Y'^i\oplus X'^{i+1}\to Y^i\oplus X^{i+1}
\end{equation}
for each integer $i\in\mathbb Z$.
It turns out that the family $F_\ast=\{F_\ast^i\}$ forms a morphism of complexes $F_\ast:\operatorname{Cone}(f')\to\operatorname{Cone}(f)$ which makes the following diagram commute (strictly):
\[
\begin{tikzcd}
Y' \ar[r] \ar[d,"v"'] & \operatorname{Cone}(f') \ar[r] \ar[d,"F_\ast"] & X'[-1] \ar[d,"u"] \\
Y \ar[r] & \operatorname{Cone}(f) \ar[r] & X[-1]
\end{tikzcd}
\quad.
\]
Actually, the construction is invariant under chain homotopies in the following sense.

\begin{lemma}
Suppose we are given a chain-homotopy commutative diagram as below:
\[
\begin{tikzcd}[column sep=4em,row sep=5ex]
X' \ar[r,"f'"] \ar[d,bend left,shift left,"u",""'{name=UR}] \ar[d,bend right,shift right,"u'"',""{name=UL}] \ar[Rightarrow,anchor=center,from=UR,to=UL,"U"'] & Y' \ar[d,bend right,shift right,"v"',""{name=VL}] \ar[d,bend left,shift left,"v'",""'{name=VR}] \ar[Rightarrow,from=VR,to=VL,"V"'] \ar[dl,Rightarrow,shorten <= 1em,shorten >= 1em,"F"'] \\
X \ar[r,"f"] & Y
\end{tikzcd}
\quad.
\]
We define a chain homotopy $F':v'f'\Rightarrow fu'$ by $F'\coloneqq fU+F+Vf'$ and write $F_\ast,F'_\ast:\operatorname{Cone}(f')\to\operatorname{Cone}(f)$ the morphisms of complexes induced by $F$ and $F'$ respectively.
Then, there is a chain homotopy $\Psi:F_\ast\Rightarrow F'_\ast$ given by
\[
\Psi^i=
\begin{pmatrix}
V^i & 0 \\ 0 & U^{i+1}
\end{pmatrix}
:Y'^i\oplus X'^{i+1}\to Y^{i-1}\oplus X^i
\quad.
\]
\end{lemma}

\begin{corollary}
\label{cor:cone-heq}
In the chain-homotopy commutative square \eqref{eq:hmtp-commsq}, suppose in addition that $u$ and $v$ are both chain homotopy equivalences.
Then the induced morphism $F_\ast$ is also a chain homotopy equivalence.
\end{corollary}

The mapping cones have further homotopy coherence.

\begin{proposition}
\label{prop:cone-hfunc}
Suppose we have a chain-homotopy commutative diagram
\begin{equation}
\label{eq:cone-hfunc:coh}
\begin{tikzcd}
X' \ar[r,"f'"] \ar[d,"u"'] & Y' \ar[d,"v"'] \ar[r,"g'"] \ar[dl,Rightarrow,"F"'] & Z' \ar[d,"w"] \ar[dl,Rightarrow,"G"'] \\
X \ar[r,"f"] & Y \ar[r,"g"] & Z
\end{tikzcd}
\end{equation}
of chain complexes such that $gf=0$ and $g'f'=0$ together with a family of morphisms $\Psi=\{\Psi^i:X'^i\to Z^{i-2}\}$ satisfying the equation
\[
d\Psi -\Psi d = g\circ F + G\circ f'
\quad.
\]
Write $\widehat g:\operatorname{Cone}(f)\to Z$ and $\widehat g':\operatorname{Cone}(f')\to Z'$ the morphisms of complexes induced by $g$ and $g'$ respectively.
Then, there is a chain homotopy depicted as below:
\begin{equation}
\begin{tikzcd}
\operatorname{Cone}(f') \ar[r,"{\widehat g'}"] \ar[d,"{F_\ast}"'] & Z' \ar[d,"w"] \ar[dl,Rightarrow,"\widehat G"'] \\
\operatorname{Cone}(f) \ar[r,"\widehat g"] & Z
\end{tikzcd}
\quad.
\end{equation}
More precisely, $\widehat G^i:\operatorname{Cone}(f')^i\to Z^{i-1}$ is presented by the matrix
\[
\widehat G=
\begin{pmatrix}
G & -\Psi
\end{pmatrix}
:Y'^i\oplus X'^{i+1}\to Z^{i-1}
\quad.
\]
\end{proposition}
\begin{proof}
Using the explicit formulas \eqref{eq:cone-nullinduce} and \eqref{eq:cone-hinduced}, we have
\[
\begin{split}
&d\widehat G+\widehat Gd-\widehat g F_\ast+w\widehat g' \\
&=
d_Z
\begin{pmatrix}
G & -\Psi
\end{pmatrix}
+
\begin{pmatrix}
G & -\Psi
\end{pmatrix}
\begin{pmatrix}
d_{Y'} & f' \\ 0 & -d_{X'}
\end{pmatrix}
-
\begin{pmatrix}
g & 0
\end{pmatrix}
\begin{pmatrix}
v & -F \\ 0 & u
\end{pmatrix}
+
\begin{pmatrix}
wg' & 0
\end{pmatrix}
\\
&=
\begin{pmatrix}
d_ZG+Gd_{Y'}-gv+wg' & -d_Z\Psi+\Psi d_{X'}+gF+Gf'
\end{pmatrix}
\quad.
\end{split}
\]
The last term vanishes by virtue of the assumption, so we obtain the result.
\end{proof}

\begin{remark}
\label{rem:fib-hfunc}
The dual of \cref{prop:cone-hfunc} also holds.
Namely, if the diagram \eqref{eq:cone-hfunc:coh} and the family $\Psi$ are given as in \cref{prop:cone-hfunc}, they induce a chain-homotopy commutative square
\[
\begin{tikzcd}
X' \ar[r,"{\overbar f'}"] \ar[d,"u"'] & \operatorname{Cone}(g')[1] \ar[d,"G_\ast"] \ar[dl,Rightarrow,"\overbar{F}"'] \\
X \ar[r,"{\overbar f}"'] & \operatorname{Cone}(g)[1]
\quad,
\end{tikzcd}
\]
where $\overbar F$ is given by the matrix
\[
\overbar F^i\coloneqq
\begin{pmatrix}
-\Psi \\ F
\end{pmatrix}
:X'^i \to Z^{i-2}\oplus Y^{i-1}
\quad.
\]
\end{remark}

\begin{corollary}
\label{cor:cone-cocone-hfunc}
Suppose we are given a chain-homotopy commutative diagram
\[
\begin{tikzcd}
X' \ar[r,"f'"] \ar[d] & Y' \ar[r,"g'"] \ar[d] \ar[dl,Rightarrow,"F"'] & Z' \ar[r,"h'"] \ar[d] \ar[dl,Rightarrow,"G"'] & W' \ar[d] \ar[dl,Rightarrow,"H"'] \\
X \ar[r,"f"'] & Y \ar[r,"g"'] & Z \ar[r,"h"'] & W
\end{tikzcd}
\]
such that
\[
gf=0\ ,\quad
hg=0\ ,\quad
g'f'=0\ ,\quad
h'g'=0
\]
together with the following data:
\begin{enumerate}[label=\upshape(\roman*)]
  \item families of morphisms $\Psi=\{\Psi^i:X'^i\to Z^{i-2}\}_i$ and $\Xi=\{\Xi^i:Y'^i\to W^{i-2}\}_i$ satisfying
\[
d\Psi-\Psi d = gF + Gf'
\ ,\quad
d\Xi-\Xi d = hG + Hg'
\quad;
\]
  \item a family of morphisms $\Gamma:\{\Gamma^i:X'^i\to W^{i-3}\}$ satisfying
\[
d\Gamma+\Gamma d = h\Psi - \Xi f'
\quad.
\]
\end{enumerate}
Then, the diagram gives rise to a chain-homotopy commutative square
\[
\begin{tikzcd}
\operatorname{Cone}(f') \ar[r] \ar[d,"F_\ast"'] & \operatorname{Cone}(h')[1] \ar[d,"{H_\ast}"] \ar[dl,Rightarrow,"\Gamma_\ast"'] \\
\operatorname{Cone}(f) \ar[r] & \operatorname{Cone}(h)[1]
\end{tikzcd}
\quad,
\]
where the chain homotopy $\Gamma_\ast$ is given by the following matrix:
\[
\Gamma_\ast^i\coloneqq
\begin{pmatrix}
-\Xi & \Gamma \\
G & -\Psi
\end{pmatrix}
:Y'^i\oplus X'^{i+1}\to W'^{i-2}\oplus Z'^{i-1}
\quad.
\]
\end{corollary}
\begin{proof}
Applying \cref{prop:cone-hfunc} with $H=0$ and $H'=0$, we obtain the following chain-homotopy commutative diagram
\[
\begin{tikzcd}
\operatorname{Cone}(f') \ar[r,"\widehat g'"] \ar[d,"F_\ast"] & Z' \ar[r,"h'"] \ar[d] \ar[dl,Rightarrow,"\widehat G"'] & W' \ar[d] \ar[dl,Rightarrow,"H"'] \\
\operatorname{Cone}(f) \ar[r,"\widehat g"'] & Z \ar[r,"h"'] & W
\end{tikzcd}
\quad.
\]
We note that both of the horizontal compositions vanish while the assumption on $\Xi$ and $\Gamma$ implies
\[
d_W
\begin{pmatrix}
\Xi & -\Gamma
\end{pmatrix}
+
\begin{pmatrix}
\Xi & -\Gamma
\end{pmatrix}
\begin{pmatrix}
d_{Y'} & f' \\
0 & -d_{X'}
\end{pmatrix}
= h\widehat G + H\widehat g'
\quad.
\]
Therefore, applying \cref{prop:cone-hfunc} again (or its dual more precisely; see \cref{rem:fib-hfunc}), one obtains the result.
\end{proof}

\subsection{Proof of \cref{theo:genus1-inv}: I}
\label{sec:invariance:invO4}

We now begin the proof of \cref{theo:genus1-inv}.
In this section, we first discuss the squares~\eqref{eq:genus1-inv:O4}; as mentioned in \cref{rem:invO4e}, we especially prove the homotopy commutativity of the left square in \eqref{eq:genus1-inv:O4}.

In the proof, we make use of \cref{cor:cone-cocone-hfunc}.
For this, we first construct a chain-homotopy commutative square in the following form:
\begin{equation}
\label{eq:O4diagram}
\begin{tikzcd}
\brcOf*{\diagCrossHRivOLWithLabel} \ar[r,"-\delta_-^{\mathsf R}"] \ar[d,"\gamma"'] & \brcOf*{\diagCrossVRivOLWithLabel} \ar[r,"\Phi^{\mathsf R}"] \ar[d,"-\omega"] \ar[dl,Rightarrow,"F"'] & \brcOf*{\diagCrossVRivOLWithLabel} \ar[r,"-\delta_+^{\mathsf R}"] \ar[d,"\omega"] \ar[dl,Rightarrow,"G"'] & \brcOf*{\diagCrossHRivOLWithLabel} \ar[d,"\gamma"] \ar[dl,Rightarrow,"H"'] \\
\brcOf*{\diagCrossHRivORWithLabel} \ar[r,"-\delta_-^{\mathsf L}"] & \brcOf*{\diagCrossVRivORWithLabel} \ar[r,"\Phi^{\mathsf L}"] & \brcOf*{\diagCrossVRivORWithLabel} \ar[r,"-\delta_+^{\mathsf L}"] & \brcOf*{\diagCrossHRivORWithLabel}
\end{tikzcd}
\quad,
\end{equation}
here $\delta_\pm^{\mathsf R}$ and $\delta_\pm^{\mathsf L}$ are the saddle operations representing the appropriate smoothing changes on the crossings, say, in the right and the left of the vertical strands; while $\Phi^{\mathsf R}$ and $\Phi^{\mathsf L}$ are the morphisms given in \eqref{eq:Phi} on those crossings.
On the other hand, $\gamma$ and $\omega$ are morphisms given as follows (see \cref{sec:univ-Kh:univBracket} for sign symbols):
\begin{equation}
\label{eq:def-gamma-omega}
\begin{gathered}
\gamma\coloneqq
\BordEquivOO\otimes\mathrm{id}
+ \BordEquivOl\otimes b'_\dagger\cra{b}
+ \BordEquivlO\otimes a'_\dagger\cra{a}
+ \BordEquivll\otimes (a'b')_\dagger\cra{(ab)}
\quad,\\
\omega\coloneqq
\BordTriId\otimes a'_\dagger\cra{b}
+\BordRTwoBarRUnit\otimes b'_\dagger\cra{b}
- \BordRTwoLCounit\otimes a'_\dagger\cra{a}
- \BordRTwoLCounitRUnit\otimes b'_\dagger\cra{a}
\quad.
\end{gathered}
\end{equation}

\begin{lemma}
\label{lem:gamma-omega-heq}
The morphisms $\gamma$ and $\omega$ given in \eqref{eq:def-gamma-omega} are chain homotopy equivalences.
\end{lemma}
\begin{proof}
It is obvious that $\gamma$ is even an isomorphism.
On the other hand, recall that Bar-Natan defined in \cite[pp.1458]{BarNatan2005} chain homotopy equivalences
\begin{equation}
\label{eq:RMove2-chaineq}
\RMove2:\brcOf*{\diagRiiRightWithLabel}
\rightleftarrows\brcOf*{\diagRiiNil}[1]:\RMoveBar2
\end{equation}
that are given as follows:
\[
\begin{gathered}
\RMove2^i\coloneqq \BordRTwoId\otimes\cra{b} - \BordRTwoCounit\otimes\cra a:
\brcOf*{\diagRiiRightWithLabel}^i
\to\brcOf*{\diagRiiNil}^{i-1}
\quad,\\
\RMoveBar2^i\coloneqq \BordRTwoId\otimes b_\dagger + \BordRTwoBarUnit\otimes a_\dagger:
\brcOf*{\diagRiiNil}^{i-1}\to
\brcOf*{\diagRiiRightWithLabel}^i
\quad.
\end{gathered}
\]
It is easily seen that $\omega$ is realized as a composition of $\RMove2$ and $\RMoveBar2$ with respect to different pairs of strands.
Hence it is also a chain homotopy equivalence.
\end{proof}

To complete the diagram \eqref{eq:O4diagram}, we define the following families of morphisms:
\begin{equation}
\label{eq:FGH}
\begin{gathered}
F^i\coloneqq(-1)^i\left(
\BordCompFFst\otimes b'_\dagger\cra{(ab)}
+ \BordCompFSnd\otimes\cra{b}
\right)
:\brcOf*{\diagCrossHRivOLWithLabel}^i\to\brcOf*{\diagCrossVRivORWithLabel}^{i-1}
\quad,\\[2ex]
G^i\coloneqq
\begin{multlined}[t]
(-1)^i\left(
\BordCompGFst\otimes a'_\dagger\cra{(ba)}
+\BordCompGSnd\otimes b'_\dagger\cra{(ba)}
+\BordCompGTrd\otimes\cra{b}
-\BordCompGFth\otimes\cra{a}
\right) \\
:\brcOf*{\diagCrossVRivOLWithLabel}^i
\to \brcOf*{\diagCrossVRivORWithLabel}^{i-1}
\quad,
\end{multlined}
\\[2ex]
H^i\coloneqq(-1)^{i+1}\left(
\BordCompHFst\otimes a'_\dagger\cra{(ba)}
+\BordCompHSnd\otimes\cra{a}
\right)
:\brcOf*{\diagCrossVRivOLWithLabel}^i
\to \brcOf*{\diagCrossHRivORWithLabel}^{i-1}
\quad.
\end{gathered}
\end{equation}
By direct computations, one can prove that the families $F=\{F^i\}_i$, $G=\{G^i\}_i$, and $H=\{H^i\}_i$ given in \eqref{eq:FGH} define chain homotopies
\[
F:\omega\delta_-^{\mathsf R}\Rightarrow -\delta_-^{\mathsf L}\gamma
\ ,\quad
G:\omega\Phi^{\mathsf R}\Rightarrow -\Phi^{\mathsf L}\omega
\ ,\quad
H:-\gamma\delta_+^{\mathsf R}\Rightarrow -\delta_+^{\mathsf L}\omega
\quad.
\]

We next construct families $\Psi_-=\{\Psi_-^i\}_i$ and $\Psi_+=\{\Psi_+^i\}_i$ of morphisms
\[
\Psi^i_+:\brcOf*{\diagCrossHRivOLWithLabel}^i\to\brcOf*{\diagCrossVRivORWithLabel}^{i-2}
\ ,\quad
\Psi^i_-:\brcOf*{\diagCrossVRivOLWithLabel}^i\to\brcOf*{\diagCrossHRivORWithLabel}^{i-2}
\]
which are coherences of the left two squares and the right ones in \eqref{eq:O4diagram} in the sense of \cref{prop:cone-hfunc}, that is, they satisfy
\begin{gather}
\label{eq:O4-coherence:minus}
d\Psi_- -\Psi_-d = \Phi^{\mathsf L}F + G\delta_-^{\mathsf R}
\quad,
\\
\label{eq:O4-coherence:plus}
d\Psi_+ - \Psi_+d = \delta_+^{\mathsf L}G + H\Phi^{\mathsf R}
\quad.
\end{gather}

\begin{lemma}
\label{lem:Psis}
We define
\begin{equation}
\label{eq:Psis:def}
\begin{gathered}
\Psi^i_-\coloneqq\BordPsiMinus\otimes\cra{(ab)}
:\brcOf*{\diagCrossHRivOLWithLabel}^i\to\brcOf*{\diagCrossVRivORWithLabel}^{i-2}
\quad,\\
\Psi^i_+\coloneqq\BordPsiPlus\otimes\cra{(ab)}
:\brcOf*{\diagCrossVRivOLWithLabel}^i\to\brcOf*{\diagCrossHRivORWithLabel}^{i-2}
\quad.
\end{gathered}
\end{equation}
Then, they satisfy the equations \eqref{eq:O4-coherence:minus} and \eqref{eq:O4-coherence:plus} respectively.
\end{lemma}
\begin{proof}
We only show the equation \eqref{eq:O4-coherence:minus}; actually \eqref{eq:O4-coherence:plus} is obtained by rotating \eqref{eq:O4-coherence:minus} in $180$ degrees around the $z$-axis with a little care about the change of the sign of the morphism $\Phi$ (see \cref{rem:Phi-prime}).

By direct computations, we obtain
\[
\begin{gathered}
d\Psi^i_-
=(-1)^i\mathbin{}\BordPsiMinus\BordDiffVRightAO\otimes a'_\dagger\cra{(ab)}
+ (-1)^i\mathbin{}\BordPsiMinus\BordDiffVRightOB\otimes b'_\dagger\cra{(ab)}
\quad,
\\
-\Psi_-^{i+1}d
=(-1)^i\mathbin{}\BordDiffHLeftAl\BordPsiMinus\otimes\cra{b}
- (-1)^i\mathbin{}\BordDiffHLeftlB\BordPsiMinus\otimes\cra{a}
\quad.
\end{gathered}
\]
On the other hand, as for the right hand side of \eqref{eq:O4-coherence:minus}, we have
\[
(\Phi^{\mathsf L})^iF^i
= \begin{multlined}[t]
(-1)^i\left(
\BordCompFFst\BordLeftPhiFstOl
-\BordCompFFst\BordLeftPhiSndOl
\right)\otimes b'_\dagger\cra{(ab)}
\\
+(-1)^i\left(
\BordCompFSnd\BordLeftPhiFstOO
-\BordCompFSnd\BordLeftPhiSndOO
\right)\otimes\cra{b}
\end{multlined}
\]
and
\[
G^i(\delta_-^{\mathsf R})^i
=
\begin{multlined}[t]
(-1)^i\mathbin{}\BordRightDeltaNll\BordCompGFst\otimes a'_\dagger\cra{(ab)}
+ (-1)^i\mathbin{}\BordRightDeltaNll\BordCompGSnd\otimes b'_\dagger\cra{(ab)}
\\
+(-1)^i\mathbin{}\BordRightDeltaNOl\BordCompGTrd\otimes\cra{b}
-(-1)^i\mathbin{}\BordRightDeltaNlO\BordCompGFth\otimes\cra{a}
\quad.
\end{multlined}
\]
Comparing the terms, we obtain
\begin{equation}
\label{eq:prf:Psis:eq-direct}
\begin{gathered}
(\text{the first term of $d\Psi^i_-$})
= (\text{the first term of $G^i(\delta^{\mathsf R}_-)^i$})
\quad,\\
(\text{the second term of $-\Psi^{i+1}_-d$})
= (\text{the fourth term of $G^i(\delta^{\mathsf R}_-)^i$})
\quad.
\end{gathered}
\end{equation}
In addition, due to the relation \ref{relK:4Tu} with respect to tubes attached to the disks in the cobordisms
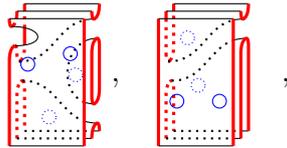
\[
\tikz{%
  \nodeTikzBox{BordCompFFst};
  \draw[blue,densely dotted] (.4,0) circle (.13);
  \draw[blue,densely dotted] (-.1,-.8) circle (.13);
  \draw[blue] (.3,.4) circle (.13);
  \draw[blue] (-.5,.2) circle (.13);
}
\ ,\quad
\tikz{%
  \nodeTikzBox{BordCompFSnd};
  \draw[blue,densely dotted] (-.3,.7) circle (.13);
  \draw[blue,densely dotted] (-.1,-.2) circle (.13);
  \draw[blue] (.3,-.5) circle (.13);
  \draw[blue] (-.5,-.5) circle (.13);
}
\quad,
\]
we also obtain the equations
\begin{equation}
\label{eq:prf:Psis:eq-4Tu}
\begin{gathered}
(\text{the second term of $d\Psi_-^i$})
=
\begin{multlined}[t]
(\text{the first term of $(\Phi^{\mathsf L})^iF^i$}) \\
+ (\text{the second term of $G^i(\delta^{\mathsf R}_-)^i$})
\quad,
\end{multlined}
\\
(\text{the second term of $-\Psi_-^{i+1}d$})
=
\begin{multlined}[t]
(\text{the first term of $(\Phi^{\mathsf L})^iF^i$}) \\
+ (\text{the third term of $G^i(\delta^{\mathsf R}_-)^i$})
\quad.
\end{multlined}
\end{gathered}
\end{equation}
Adding \eqref{eq:prf:Psis:eq-direct} and \eqref{eq:prf:Psis:eq-4Tu}, we obtain the result.
\end{proof}

Finally, we apply \cref{cor:cone-cocone-hfunc} to the diagram \ref{eq:O4diagram}.
In fact, we have
\[
\delta^{\mathsf R}_+\Psi_-
= \BordRightDeltaPll\BordPsiMinus\otimes\cra{ab}
= \BordPsiPlus\BordLeftDeltaNOO\otimes\cra{ab}
= \Psi_+\delta^{\mathsf L}_-
\quad.
\]
Hence, all the assumptions in \cref{cor:cone-cocone-hfunc} are satisfied with respect to $\Psi_-$, $\Psi_+$, and $\Gamma=0$.
Therefore, a chain homotopy in the left square of \eqref{eq:genus1-inv:O4} is induced.
Moreover, by \cref{lem:gamma-omega-heq}, all the vertical arrows in \eqref{eq:O4diagram} are chain homotopy equivalences.
It then follows from \cref{cor:cone-heq} that the vertical arrows in \eqref{eq:genus1-inv:O4} are chain homotopy equivalences.
This proves the homotopy commutativity of the squares in \eqref{eq:genus1-inv:O4}.

\subsection{Proof of \cref{theo:genus1-inv}: II}
\label{sec:invariance:invO5}

In contrast to the arguments in the previous section, the homotopy commutativity of \cref{eq:genus1-inv:O5} is relatively easy.
In fact, it is a consequence of the following lemma, whose proof is left to the reader as it is mostly straightforward.

\begin{lemma}
\label{lem:O5-expanded}
The following diagram commutes:
\begin{equation}
\label{eq:O5-expanded}
\begin{tikzcd}
\brcOf*{\diagRiiNil}[1] \ar[r,"\RMoveBar2^\circ"] \ar[d,"\RMoveBar2"'] & \brcOf*{\diagRiiLeftUpWithSign} \ar[d,"\widehat\Phi_a"] \\
\brcOf*{\diagRiiRightUpWithSign} \ar[r,"\widehat\Phi_b"] & \brcOf*{\diagRvFTwWithSign}[1]
\end{tikzcd}
\quad,
\end{equation}
here the morphisms $\RMoveBar2$ and $\RMoveBar2^\circ$ are the ones given in \eqref{eq:RMove2-chaineq} while the latter is applied to $180$-degree-rotated diagrams.
\end{lemma}

We write $\RMove2^\circ$ the chain homotopy inverse to $\RMoveBar2^\circ$ in the diagram \eqref{eq:O5-expanded} described in \eqref{eq:RMove2-chaineq}.
Using the chain homotopy $\mathrm{id}\Rightarrow\RMove2^\circ\RMoveBar2^\circ$, one can define a chain homotopy $H$ as in the diagram below:
\[
\begin{tikzcd}
\brcOf*{\diagRiiLeftUpWithSign} \ar[r,"\widehat\Phi_a"] \ar[d,"\RMoveBar2\RMove2^\circ"'] & \brcOf*{\diagRvFTwWithSign}[1] \ar[d,equal] \ar[dl,Rightarrow,"H"'] \\
\brcOf*{\diagRiiRightUpWithSign} \ar[r,"\widehat\Phi_b"] & \brcOf*{\diagRvFTwWithSign}[1]
\end{tikzcd}
\quad.
\]
Since the morphism $\RMoveBar2$ is also a chain homotopy equivalence, the left vertical arrow above is a chain homotopy equivalence.
This verifies the homotopy commutativity of \cref{eq:genus1-inv:O5}.
Therefore, the proof of \cref{theo:genus1-inv} is finally completed.

\section{Universal Khovanov complex for singular tangles}
\label{sec:UKHsing}

Using the invariance of the genus-one morphism $\widehat\Phi$ proved in \cref{sec:invariance}, we can now extend the universal Khovanov complex to singular tangles so that an analogue of Vassiliev skein relation holds.

\subsection{Definition}
\label{sec:UKHsing:def}

We first extend the universal bracket complex $\brcOf{\blank}$ to singular tangle diagrams.
For a singular tangle diagram $D$, we denote by $c^\sharp(D)$ the set of double points of $D$.
We call each subset $\mathfrak r\subset c^\sharp(D)$ a \emph{resolution scheme} of $D$ and denote by $|\mathfrak r|$ the cardinality of $\mathfrak r$.
For each resolution scheme $\mathfrak r$, we obtain an ordinary tangle diagram (i.e.~without double points) $D_{\mathfrak r}$ which is identical to $D$ except near the double points where we have
\[
\begin{array}{c|c|c}
  D & D_{\mathfrak r}\;\text{($b\notin\mathfrak r$)} & D_{\mathfrak r}\;\text{($b\in\mathfrak r$)} \\\hline\rule{0pt}{5ex}
  \diagSingUpWith{b} & \diagCrossNegUpWith{b^-} & \diagCrossPosUpWith{b^+}
\end{array}
\]
for each $b\in c^\sharp(D)$.
In particular, a checkerboard coloring $\chi$ on $(\mathbb R\times[0,1])\setminus D$ induces that on $D_{\mathfrak r}$ for each resolution scheme $\mathfrak r$.
Hence, if $b\notin\mathfrak r$, the genus-one morphism yields a morphism of chain complexes
\[
\widehat\Phi_{\mathfrak r,b}:\brcOf{D_{\mathfrak r}^\chi}\to \brcOf{D_{\mathfrak r\cup\{b\}}^\chi}[1]
\quad.
\]
We now define $\brcOf{D^\chi}$ as a graded object in $\Cob(\partial^-D^\chi,\partial^+D^\chi)$ by
\[
\brcOf{D^\chi}
\coloneqq\bigoplus_{\mathfrak r\subset c^\sharp(D)}\brcOf{D^\chi_{\mathfrak r}}[2|\mathfrak r|]\otimes E_{\mathfrak r}
\quad.
\]
The differential $d_D=\{d^i_D\}_i$ consists of $d^i_D:\brcOf{D^\chi}^i\to\brcOf{D^\chi}^{i+1}$ which is componentwisely given by
\[
\begin{multlined}[t]
d_{D_{\mathfrak r}}^{i+2|\mathfrak r|}\otimes{\mathrm{id}_{E_{\mathfrak r}}}+\sum_{b\in c^\sharp(D)\setminus\mathfrak r}\widehat\Phi_{\mathfrak r,b}\otimes(\wedge b): \\
\brcOf{D_{\mathfrak r}^\chi}[2|\mathfrak r|]^i\otimes E_{\mathfrak r}
\to (\brcOf{D_{\mathfrak r}^\chi}[2|\mathfrak r|]^{i+1}\otimes E_{\mathfrak r})\oplus\bigoplus_{b\in c^\sharp(D)\setminus\mathfrak r}\brcOf{D_{\mathfrak r\cup\{b\}}^\chi}[2|\mathfrak r|+2]^{i+1}\otimes E_{\mathfrak r\cup\{b\}}
\end{multlined}
\]
for each resolution scheme $\mathfrak r$.

We see that the complex $\brcOf{D}$ is also realized as an iterated mapping cone.
For this, fix a double point $b\in c^\sharp(D)$ and let $D^{(+)}$ and $D^{(-)}$ be diagrams obtained from $D$ by resolving $b$ into positive and negative crossings respectively.
The set $c^\sharp(D_\pm)$ is then identified with $c^\sharp(D)\setminus\{b\}$ so that, for each resolution scheme $\mathfrak r\subset c^\sharp(D)\setminus\{b\}$, we have $D^{(-)}_{\mathfrak r}=D_{\mathfrak r}$ and $D^{(+)}_{\mathfrak r}=D_{\mathfrak r\cup\{b\}}$.
Under these identifications, one can see that the inclusion and the projection yield the following exact sequence of morphisms of chain complexes in $\Cob(\partial^-D^\chi,\partial^+D^\chi)$:
\begin{equation}
\label{eq:singbrc-incproj}
\brcOf{D^{(+),\chi}}[2]\xrightarrow{b_\dagger}\brcOf{D^\chi}\twoheadrightarrow \brcOf{D^{(-),\chi}}
\quad.
\end{equation}
On the other hand, for each $\mathfrak r\subset c^\sharp(D)\setminus\{b\}$, we have genus-one morphism $\brcOf{D^{(-),\chi}_{\mathfrak r}}\to\brcOf{D^{(+),\chi}_{\mathfrak r}}[1]$.
As $\mathfrak r$ varies in resolution schemes on $D^{(-)}$, it turns out that it defines a morphism
\begin{equation}
\label{eq:singular-genus-one}
\widehat\Phi:\brcOf{D^{(-),\chi}}\to\brcOf*{D^{(+),\chi}}[1]
\end{equation}
which we again call \emph{genus-one morphism}.

\begin{proposition}
\label{prop:vassiliev-cone}
In the situation above, there is an isomorphism
\begin{equation}
\label{eq:singbrc-cone}
\brcOf{D^\chi}
\cong \operatorname{Cone}\left(
\brcOf{D^{(-),\chi}}\xrightarrow{\widehat\Phi}\brcOf{D^{(+),\chi}}[1]
\right)[1]
\end{equation}
in the category $\mathbf{Ch}^{\mathsf b}(\Cob(\partial^-D^\chi,\partial^+D^\chi))$ so that the associated exact sequence is exactly \eqref{eq:singbrc-incproj}.
\end{proposition}

The proof is almost identical to \cref{prop:coneVSbracket} and so omitted.

We now extend the universal Khovanov complex to singular tangles diagrams by normalizing the degree of the universal bracket complex.
The argument is almost parallel to the case of ordinary tangles.

\begin{definition}
Let $D$ be a singular tangle diagram with $n_-$ negative crossings and $n_\times$ double points.
We write $\chi_w$ the checkerboard coloring with negatively unbounded white component.
Then, we set
\begin{equation}
\label{eq:singKh-def}
\KhOf*{D}\coloneqq \brcOf*{D^{\chi_w}}[-n_- - 2n_\times]
\in\Cob(\partial^-D^{\chi_w},\partial^+D^{\chi_w})
\end{equation}
and call it the \emph{universal Khovanov complex} of $D$.
\end{definition}

It is obvious that, if $D$ has no double point, the definition agrees with the ordinary universal Khovanov complex.
In \cref{sec:UKHsing:inv}, we see that $\KhOf*{\blank}$ defines an invariant of singular tangles up to chain homotopies.

\begin{corollary}
\label{cor:vassiliev-cone-UKH}
In the same situation as \cref{prop:vassiliev-cone}, there is an isomorphism
\[
\KhOf{D}
\cong \operatorname{Cone}\left(
\KhOf{D^{(-)}} \xrightarrow{\widehat\Phi}\KhOf{D^{(+)}}
\right)
\quad.
\]
\end{corollary}

We may think of \cref{cor:vassiliev-cone-UKH} as a categorified analogue of \emph{Vassiliev skein relation}.
Namely, it gives rise to a distinguished triangle
\begin{equation}
\label{eq:Kh-Vassiliev-triangle}
\cdots
\to \KhOf*{\diagCrossNegUp}
\xrightarrow{\widehat\Phi}\KhOf*{\diagCrossPosUp}
\to \KhOf*{\diagSingUp}
\to \KhOf*{\diagCrossNegUp}[-1]
\xrightarrow{\widehat\Phi} \cdots
\end{equation}
in the homotopy category (with the standard triangulated structure).
In fact, if we write $\lbrack\blank\rbrack$ the image of the universal Khovanov complex $\KhOf{\blank}$ in the $K$-group, \eqref{eq:Kh-Vassiliev-triangle} yields the equation
\[
\left\lbrack\diagSingUp\right\rbrack
= \left\lbrack\diagCrossPosUp\right\rbrack
- \left\lbrack\diagCrossNegUp\right\rbrack
\quad,
\]
which is exactly the Vassiliev skein relation.

\begin{example}
Evaluating the sequence \eqref{eq:Kh-Vassiliev-triangle} with the Euler-graded TQFT associated with the Frobenius algebra $k[x]/(x^2)$, we obtain the following long exact sequence of Khovanov homologies with coefficients in $k$:
\[
\begin{tikzcd}[column sep=.9em]
\cdots \ar[r] & \operatorname{Kh}^{i,j}\left(\diagCrossNegUp\;;k\right) \ar[r,"\widehat\Phi_\ast"] & \operatorname{Kh}^{i,j}\left(\diagCrossPosUp\;;k\right) \ar[r] & \operatorname{Kh}^{i,j}\left(\diagSingUp\;;k\right) \ar[dl,out=-20,in=160] & \\
&& \operatorname{Kh}^{i+1,j}\left(\diagCrossNegUp\;;k\right) \ar[r,"\widehat\Phi_\ast"] & \operatorname{Kh}^{i+1,j}\left(\diagCrossPosUp\;;k\right) \ar[r] &
\cdots\quad.
\end{tikzcd}
\]
In case $k$ is a field, one can recover the Vassiliev skein relation for the Jones polynomial by taking the graded Euler characteristics.
\end{example}

\subsection{Composition formulas}
\label{sec:UKHsing:composition}

As seen in \cref{prop:TangleOtimes} and \cref{prop:TangleAst}, the universal bracket complexes behave well for compositions and tensor products of tangle diagrams.
Actually, there are analogous isomorphisms for singular tangle diagrams.
We define compositions and tensor products of singular tangle diagrams in the same manner as the non-singular case.

\begin{theorem}
\label{theo:singtang-comp}
Let $D$ be the composition of two singular tangle diagrams, say $D'$ and $D''$.
For a checkerboard coloring $\chi$ on $(\mathbb R\times[0,1])\setminus D$, let us write $\chi'$ and $\chi''$ respectively the induced coloring on the complements of $D'$ and $D''$.
Then, there is an isomorphism
\begin{equation}
\label{eq:singtang-comp}
\brcOf{D^\chi}\cong\brcOf{D'^{\chi'}}\ast\brcOf{D''^{\chi''}}
\end{equation}
in the category $\mathbf{Ch}^{\mathsf b}(\Cob(\partial^-D^\chi,\partial^+D^\chi))=\mathbf{Ch}^{\mathsf b}(\Cob(\partial^-{D'}^{\chi'},\partial^+{D''}^{\chi''}))$.
\end{theorem}
\begin{proof}
We may identify the double points in $D'$ and $D''$ with those in $D$; in other words, $c^\sharp(D)=c^\sharp(D')\amalg c^\sharp(D'')$.
For a resolution scheme $\mathfrak r\subset c^\sharp(D)$, we write $\mathfrak r'\coloneqq\mathfrak r\cap c^\sharp(D')$ and $\mathfrak r''\coloneqq\mathfrak r\cap c^\sharp(D'')$.
Then, by \cref{prop:TangleAst} implies that there is an isomorphism
\begin{equation}
\label{eq:prf:singtang-comp:resolcomp}
\brcOf*{D^\chi_{\mathfrak r}}
\cong \brcOf{D''^{\chi''}_{\mathfrak r''}}\ast\brcOf{D'^{\chi'}_{\mathfrak r'}}
\quad.
\end{equation}
In addition, for any pair of integers $(p,q)$, the morphism
\[
(-1)^{iq}\ast\mathrm{id}:
\brcOf{D''^{\chi''}_{\mathfrak r''}}^{i-p}\ast\brcOf{D'^{\chi'}_{\mathfrak r'}}^{j-q}
\to \brcOf{D''^{\chi''}_{\mathfrak r''}}^{i-p}\ast\brcOf{D'^{\chi'}_{\mathfrak r'}}^{j-q}
\]
defines an isomorphism
\begin{equation}
\label{eq:prf:singtang-comp:shift}
\brcOf{D''^{\chi''}_{\mathfrak r''}}[p]\ast\brcOf{D'^{\chi'}_{\mathfrak r'}}[q]
\cong (\brcOf{D''^{\chi''}_{\mathfrak r''}}\ast\brcOf{D'^{\chi'}_{\mathfrak r'}})[p+q]
\quad.
\end{equation}
Thanks to the identifications \eqref{eq:prf:singtang-comp:resolcomp} and \eqref{eq:prf:singtang-comp:shift}, we obtain \eqref{eq:singtang-comp} as an isomorphism of graded objects in $\Cob(\partial^-D^\chi,\partial^+D^\chi)$.
Furthermore, for each $b\in c^\sharp(D)\setminus\mathfrak r$, the genus-one morphism $\widehat\Phi_{\mathfrak r,b}:\brcOf{D^\chi_{\mathfrak r}}\to\brcOf{D^\chi_{\mathfrak r\cup\{b\}}}[1]$ with respect to $b$ is given by
\[
\widehat\Phi_{\mathfrak r,b}=
\begin{cases}
\mathrm{id}\ast\widehat\Phi_{\mathfrak r',b}:\brcOf{D''^{\chi''}_{\mathfrak r''}}\ast\brcOf{D'^{\chi'}_{\mathfrak r'}}\to \brcOf{D''^{\chi''}_{\mathfrak r''}}\ast\brcOf{D'^{\chi'}_{\mathfrak r'\cup\{b\}}}[1] & b\in c^\sharp(D')\quad, \\[1ex]
\widehat\Phi_{\mathfrak r'',b}\ast\mathrm{id}:\brcOf{D''^{\chi''}_{\mathfrak r''}}\ast\brcOf{D'^{\chi'}_{\mathfrak r'}}\to \brcOf{D''^{\chi''}_{\mathfrak r''\cup\{b\}}}[1]\ast\brcOf{D'^{\chi'}_{\mathfrak r'}}& b\in c^\sharp(D'')\quad.
\end{cases}
\]
Using this formula, one can easily verify that the differentials on the complexes in \eqref{eq:singtang-comp} agree with each other.
Hence, the result follows.
\end{proof}

The same argument also shows the following.

\begin{theorem}
\label{theo:singtang-disj}
Let $D$ be the tensor product of two singular tangle diagrams, say $D'$ and $D''$.
For a checkerboard coloring $\chi$ on $(\mathbb R\times[0,1])\setminus D$, we write $\chi'$ and $\chi''$ respectively the induced coloring on the complements of $D'$ and $D''$.
Then, there is an isomorphism
\[
\brcOf{D^\chi}
\cong \brcOf{D'^{\chi'}}\otimes\brcOf{D''^{\chi''}}
\]
in the category $\mathbf{Ch}^{\mathsf b}(\Cob(\partial^-D^\chi,\partial^+D^\chi))=\mathbf{Ch}^{\mathsf b}(\Cob(\partial^-D'^{\chi'}\otimes\partial^-D''^{\chi''},\partial^+D'^{\chi'}\otimes\partial^+D''^{\chi''}))$.
\end{theorem}

\subsection{Invariance}
\label{sec:UKHsing:inv}

Using the results obtained in \cref{sec:UKHsing:composition}, we now prove that the extended universal Khovanov complex yields an invariant for singular tangles.
To be more precise, by an (oriented) \emph{singular tangle}, we mean a compact oriented immersed submanifold $T\subset\mathbb R^2\times[0,1]$ of dimension $1$ such that
\begin{enumerate}[label=\upshape(\alph*)]
  \item it has only finitely many transverse double points as singularities; and
  \item it is a neat submanifold near the boundary $\mathbb R^2\times\{0,1\}$.
\end{enumerate}
We always assume a singular tangle $T$ to be ``generic'' so that the image of $T$ under the projection $\mathbb R^2\times[0,1]\to\mathbb R\times[0,1]$ defines a singular tangle diagram $D$; in this case, we call $D$ \emph{the} diagram of $T$.
We also consider isotopies between them in the same way as the non-singular case (see \cref{sec:univ-Kh:univKh}).

\begin{theorem}
\label{theo:Khsing-inv}
The assignment $D\mapsto\KhOf{D}$ is invariant under the local moves \eqref{eq:moveSing} in addition to the Reidemeister moves up to chain homotopy equivalences.
Consequently, it defines an ambient isotopy invariant for singular tangles.
\end{theorem}
\begin{proof}
We first show that $\KhOf{\blank}$ defines an isotopy invariant of tangles; in other words, we show that, if $D$ and $D'$ are the diagrams of isotopic singular tangles, then there is a chain homotopy equivalence $\KhOf{D}\simeq\KhOf{D'}$.
Since chain homotopy equivalences compose, we may assume that $D$ and $D'$ are connected by a single elementary move; that is, one of the moves \eqref{eq:moveSing} and Reidemeister moves.
Furthermore, by virtue of \cref{theo:singtang-comp} and \cref{theo:singtang-disj}, we are reduced to the case where $D$ and $D'$ are exactly the local tangles involved with the move.
For Reidemeister moves, the result is nothing but the invariance of the universal Khovanov complexes for ordinary tangles.
On the other hand, for moves \eqref{eq:moveSing}, the result is exactly \cref{theo:genus1-inv}.

Now, the isomorphism in the statement is exactly \cref{cor:vassiliev-cone-UKH}.
To verify the last statement, set $k=\mathbb F_2$ and let $Z:\Cob(\varnothing,\varnothing)\to\mathbf{Mod}_{\mathbb F_2}$ be the $2$-dimensional TQFT associated with the Frobenius algebra $\mathbb F_2[x]/(x^2)$.
One can easily verify that the image $Z(\widehat\Phi)$ of the genus-one morphism coincides with the \emph{genus-one map} introduced in \cite[Section~3.2]{ItoYoshida2019}.
Combining this observation with \cref{cor:vassiliev-cone-UKH}, we can conclude that, for every singular link diagram $D$, the image $Z(\KhOf{D})$ is isomorphic to the complex constructed in \cite[Section~4.2]{ItoYoshida2019}.
This completes the proof.
\end{proof}

\subsection{The FI relation}
\label{sec:UKHsing:FI}

We now prove that our extension of the universal Khovanov complex satisfies a categorified version of the FI relation.
Note that, as seen in \cref{fig:FI-Vassiliev}, the FI relation comes from a comparison between crossing change and two-fold Reidemeister moves of type~\RomNum1.
Thus, one can realize it as a \emph{commutativity} with Reidemeister moves of type~\RomNum1.
We here consider the universal bracket complex for a technical reason.

In order to formulate the commutativity, we recall the morphism on the universal bracket complexes associated with the Reidemeister move of type~\RomNum1.
Namely, according to Bar-Natan~\cite{BarNatan2005}, we have the following chain homotopy equivalences:
\begin{equation}
\label{eq:RI-morcob}
\begin{gathered}
\RMoveBar1^-\coloneqq\BordROneNBar\otimes c^-_\dagger:\brcOf*{\diagFiNil^\chi}\to\brcOf*{\diagFiNegWithLabel^\chi}[-1]
\quad,\\[1ex]
\RMoveBar1^+\coloneqq\BordROnePBarFst\otimes\mathrm{id}-\BordROnePBarSnd\otimes\mathrm{id}:\brcOf*{\diagFiNil^\chi}\to\brcOf*{\diagFiPosWithLabel^\chi}
\quad.
\end{gathered}
\end{equation}
Using these morphisms, we prove the following.

\begin{theorem}
\label{theo:Phi-comm-RI}
Suppose $D$, $D^{(+)}$, and $D^{(-)}$ are singular tangle diagrams which differ only by local pictures as in \cref{fig:RI-pic}.
\begin{figure}[t]
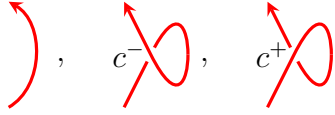

\centering
\[
\diagFiNil
\ ,\quad
\diagFiNegWithLabel
\ ,\quad
\diagFiPosWithLabel
\]
\caption{Local pictures for $D$ (left), $D^{(+)}$ (middle), and $D^{(-)}$ (right)}.
\label{fig:RI-pic}
\end{figure}
Then, for arbitrary checkerboard coloring $\chi$ on the complement of $D$, the following diagram commutes strictly:
\[
\begin{tikzcd}[column sep=.5em]
& \brcOf*{D^\chi} \ar[dl,"\RMoveBar1^-"'] \ar[dr,"\RMoveBar1^+"] & \\
\brcOf*{D^{(-),\chi}}[-1] \ar[rr,"\widehat\Phi"] && \brcOf*{D^{(+),\chi}}
\end{tikzcd}
\quad.
\]
\end{theorem}
\begin{proof}
By virtue of \cref{theo:singtang-comp} and \cref{theo:singtang-disj}, we are reduced to the case where $D$, $D^{(-)}$, and $D^{(+)}$ are exactly the local pictures in \cref{fig:RI-pic}.
Hence, by the relation~\ref{relK:4Tu} in \cref{sec:univ-Kh:K}, it turns out that the genus-one morphism $\widehat\Phi$ is given as follows:
\[
\widehat\Phi=\BordFIPhiSndVar\otimes\cra{c^-}-\BordFIPhiFst\otimes\cra{c^-}
:\brcOf*{\diagFiNegWithLabel^\chi}\to\brcOf*{\diagFiPosWithLabel^\chi}[1]
\quad.
\]
Therefore, we can verify the result by observing the following diffeomorphisms of cobordisms:
\[
\BordROneNBar\BordFIPhiSndVar
\cong \BordROnePBarFst
\ ,\quad
\BordROneNBar\BordFIPhiFst
\cong \BordROnePBarSnd
\quad.
\]
\end{proof}

\begin{corollary}
\label{cor:Ubrc-FI}
Suppose $D$ is a singular tangle diagram one of whose double point is of the form in \cref{fig:dblpt-FI}.
\begin{figure}[t]
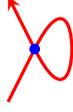

\centering
\diagFiSing
\caption{A double point in FI relation.}
\label{fig:dblpt-FI}
\end{figure}
Then, for every checkerboard coloring $\chi$, the universal bracket complex $\brcOf{D^\chi}$ is contractible, i.e.~the identity is null-homotopic.
Consequently, we also have $\KhOf{D}\simeq 0$.
\end{corollary}
\begin{proof}
The assertion is equivalent to saying that the genus-one morphism of the form
\[
\widehat\Phi:\brcOf*{\diagFiNegWithLabel^\chi}\to\brcOf*{\diagFiPosWithLabel^\chi}[1]
\]
is a chain homotopy equivalence, which is exactly \cref{theo:Phi-comm-RI}.
\end{proof}

\subsection{Examples}
\label{sec:UKHsing:examples}

We conclude the paper with some examples of our invariant for singular links.
Recall that Khovanov \cite{Khovanov2006} classified Frobenius algebras which give rise to link invariants.
Namely, for a fixed coefficient ring $k$, and for two elements $h,t\in k$, we define a $k$-algebra $C_{h,t}\coloneqq k[x]/(x^2-hx-t)$ with a Frobenius algebra structure given by
\[
\begin{gathered}
\Delta(1)\coloneqq 1\otimes x + x\otimes 1 - h x\otimes x
\ ,\quad \Delta(x)\coloneqq x\otimes x + t 1\otimes 1
\ ,\\
\varepsilon(1)\coloneqq 0\ ,\quad \varepsilon(x)\coloneqq 1
\quad.
\end{gathered}
\]
If we denote by $Z_{h,t}:\mathbf{Cob}_2(\varnothing,\varnothing)\to\mathbf{Mod}_k$ the associated TQFT, then it turns out that it induces a symmetric monoidal $k$-linear functor $\Cob(\varnothing,\varnothing)\to\mathbf{Mod}_k$, which is again written $Z_{h,t}$ by abuse of notation.
Then, for a singular link diagram $D$, we define
\[
\KhOf{D}_{h,t}
\coloneqq Z_{h,t}(\KhOf{D})
\quad.
\]
As a consequence of \cref{theo:Khsing-inv}, the homology of the complex $\KhOf{D}_{h,t}$ is an invariant of the singular link defined by $D$.

We compute the complex $\KhOf{\blank}_{h,t}$ for the following three diagrams:
\[
D^{(1)}=
\begin{tikzpicture}[baseline=(current bounding box.center)]
\draw[red,very thick,-stealth] (0.29,0) arc(0:180:1);
\draw[red,very thick,-stealth] (-1.71,0) arc(180:360:1);
\fill[white] (0,-.71) circle(.15);
\draw[red,very thick,-stealth] (-.29,0) arc(180:0:1);
\draw[red,very thick,-stealth] (1.71,0) arc(0:-180:1);
\fill[blue] (0,.71) circle(.1);
\end{tikzpicture}
\ ,\quad
D^{(2)}=
\begin{tikzpicture}[baseline=(current bounding box.center)]
\draw[red,very thick,-stealth] (-.29,0) arc(180:0:1);
\draw[red,very thick,-stealth] (1.71,0) arc(0:-180:1);
\fill[white] (0,-.71) circle(.15);
\draw[red,very thick,-stealth] (0.29,0) arc(0:180:1);
\draw[red,very thick,-stealth] (-1.71,0) arc(180:360:1);
\fill[blue] (0,.71) circle(.1);
\end{tikzpicture}
\ ,\quad
D^{(3)}=
\begin{tikzpicture}[baseline=(current bounding box.center)]
\draw[red,very thick,-stealth] (0.29,0) arc(0:180:1);
\draw[red,very thick,-stealth] (-1.71,0) arc(180:360:1);
\draw[red,very thick,-stealth] (-.29,0) arc(180:0:1);
\draw[red,very thick,-stealth] (1.71,0) arc(0:-180:1);
\fill[blue] (0,-.71) circle(.1);
\fill[blue] (0,.71) circle(.1);
\end{tikzpicture}
\]
Unwinding the definition, one sees that the chain complex $\brcOf{D^{(1)}}$ is isomorphic to the cochain complex associated to the following skew-commutative diagram in $\Cob(\varnothing,\varnothing)$:
\begin{equation}
\label{eq:hopf-sing-cplx}
\begin{tikzcd}
\diagHopfOO \ar[r,"\delta"] \ar[d,"\delta"'] & \diagHopflO \ar[r,"\Phi"] \ar[d,"\delta"] & \diagHopflO \ar[r,"\delta"] \ar[d,"\delta"] & \diagHopfOO \ar[d,"\delta"] \\
\diagHopfOl \ar[r,"-\delta"] & \diagHopfll \ar[r,"-\Phi"] & \diagHopfll \ar[r,"-\delta"] & \diagHopfOl
\end{tikzcd}
\quad,
\end{equation}
here the morphisms with label $\delta$ are saddles while the ones with $\Phi$ are the morphisms given in \eqref{eq:Phi}.
Applying the functor $Z_{h,t}$, we obtain a skew-commutative diagram below:
\begin{equation}
\label{eq:hopf-sing-frobcplx}
\begin{tikzcd}
C_{h,t}\otimes C_{h,t} \ar[r,"\mu"] \ar[d,"\mu"'] & C_{h,t} \ar[r,"0"] \ar[d,"\Delta"] & C_{h,t} \ar[r,"\Delta"] \ar[d,"\Delta"] & C_{h,t} \ar[d,"\mu"] \\
C_{h,t} \ar[r,"-\Delta"] & C_{h,t}\otimes C_{h,t} \ar[r,"-\varphi"] & C_{h,t}\otimes C_{h,t} \ar[r,"-\mu"] & C_{h,t}
\end{tikzcd}
\quad,
\end{equation}
here $\mu$ and $\Delta$ are the multiplication and the comultiplication of the Frobenius algebra $C_{h,t}$ respectively and $\varphi:C_{h,t}\otimes C_{h,t}\to C_{h,t}\otimes C_{h,t}$ is given by
\[
\varphi(a\otimes b)
\coloneqq a\otimes xb - ax\otimes b
\quad.
\]
It turns out that the bottom row of \eqref{eq:hopf-sing-frobcplx} is exact so that we obtain an isomorphism
\begin{equation}
\label{eq:UKH-singhopfn}
H^i\left(\KhOf*{D^{(1)}}_{h,t}\right)\cong
\begin{cases}
\operatorname{ker}\mu & i=-3\ ,\\
\operatorname{coker}\Delta & i=0\ ,\\
0 & \text{otherwise}\ .
\end{cases}
\end{equation}
A similar computation also shows that
\begin{equation}
\label{eq:UKH-singhopfp}
H^i\left(\KhOf*{D^{(2)}}_{h,t}\right)\cong
\begin{cases}
\operatorname{ker}\mu & i=-1\ ,\\
\operatorname{coker}\Delta & i=2\ ,\\
0 & \text{otherwise}\ .
\end{cases}
\end{equation}
Finally, by virtue of \cref{cor:vassiliev-cone-UKH}, we have a long exact sequence
\[
\cdots\to H^{i-1}\left(\KhOf{D^{(3)}}_{h,t}\right)
\to H^i\left(\KhOf{D^{(1)}}_{h,t}\right)
\xrightarrow{\widehat\Phi} H^i\left(\KhOf{D^{(2)}}_{h,t}\right)
\to H^i\left(\KhOf{D^{(3)}}_{h,t}\right)
\to \cdots
\]
of $k$-modules.
Note that, as seen in \eqref{eq:UKH-singhopfn} and \eqref{eq:UKH-singhopfp}, the cohomology groups of $\KhOf{D^{(1)}}_{h,t}$ and $\KhOf{D^{(2)}}_{h,t}$ are direct summands of the free $k$-module $C_{h,t}\otimes C_{h,t}$ and hence all projective.
Therefore, we obtain an isomorphism
\[
H^i\left(\KhOf*{D^{(3)}}_{h,t}\right)\cong
\begin{cases}
\operatorname{ker}\mu & i=-4\ ,\\
\operatorname{ker}\mu\oplus\operatorname{coker}\Delta & i=-1\ ,\\
\operatorname{coker}\Delta & i=2\ ,\\
0 & \text{otherwise}.
\end{cases}
\]

\section*{Acknowledgments}
We are deeply grateful to Professor Toshitake Kohno whose guidance and supports were invaluable during our study.
We would like to thank Professor Jun Murakami for his kind encouragement.
We are also indebted to Professor J.~Scott Carter for helping us communicate with expert researchers.
Professor Louis H.~Kauffman gave us constructive feedbacks in MSCS Quantum Topology Seminar.

\bibliographystyle{plain}
\bibliography{../reference}

\end{document}